\documentclass{article}
\usepackage[a4paper, total={6in, 8in}]{geometry}
\usepackage[utf8]{inputenc}
\usepackage{dsfont}
\usepackage{mathtools}
\usepackage{graphicx}
\usepackage{amsmath}
\usepackage{amssymb}
\usepackage{booktabs}
\usepackage{mathtools} 
\usepackage{enumitem}
\usepackage[color=yellow]{todonotes}
\usepackage{cancel}
\usepackage[normalem]{ulem}
\usepackage{environ}
\usepackage{readarray} 
\usepackage[bb=boondox]{mathalfa}
\usepackage[hypertexnames=false,breaklinks,colorlinks=true,citecolor=blue,linkcolor=blue,urlcolor=blue!50!black]{hyperref}
\usepackage{doi}

\usepackage{caption}
\usepackage{subcaption}

\usepackage{tikz}
\usepackage{tikz-qtree}
\usepackage{forest}

\usetikzlibrary{shapes.misc,calc,decorations.pathmorphing,shapes}
\usetikzlibrary{arrows.meta}
\usetikzlibrary{positioning}
\usetikzlibrary{shapes,fit} 

    \tikzstyle{vertex}=[circle, fill, inner sep=0pt, minimum size=6pt] 
    \tikzset{cross/.style={cross out, draw=black,line width= .4ex, minimum size=5pt, inner sep=0pt, outer sep=0pt}, cross/.default={1pt}}

\usepackage[linesnumbered,boxed,noend]{algorithm2e}
\SetKwInOut{Input}{Input}
\SetKwInOut{Output}{Output}

\usepackage{amsthm}

\newtheorem{theorem}{Theorem}
\newtheorem{cor}[theorem]{Corollary}
\newtheorem{lemma}[theorem]{Lemma}

\newtheorem{obs}[theorem]{Observation}
\newtheorem*{claim}{Claim}

\newcommand{\R}{\mathds{R}}
\newcommand{\N}{\mathds{N}}
\newcommand{\Z}{\mathds{Z}}
\newcommand{\Q}{\mathds{Q}}

\renewcommand{\S}{\ensuremath{\mathcal{S}}}
\newcommand{\NN}{\ensuremath{\mathcal{N}}}

\newcommand{\ZZ}{\ensuremath{\mathcal{Z}}}
\newcommand{\D}{\ensuremath{\mathcal{D}}}
\newcommand{\C}{\ensuremath{\mathcal{C}}}

\newcommand{\U}{\ensuremath{\mathcal{U}}}
\newcommand{\V}{\ensuremath{\mathcal{V}}}

\newcommand{\SBB}{\ensuremath{\mathcal{T}}\xspace}

\newcommand{\T}{^\top}

\newcommand{\define}{\coloneqq}

\DeclarePairedDelimiterX{\card}[1]{\lvert}{\rvert}{#1}

\DeclareMathOperator{\anc}{anc}

\DeclareMathOperator{\poly}{poly}
\DeclareMathOperator{\range}{range}
\DeclareMathOperator{\Prob}{Pr}

\begin{document}
\title{Sub-Exponential Lower Bounds for Branch-and-Bound with General Disjunctions via Interpolation}
\author{Max Gl\"aser
\and
Marc E. Pfetsch
}


\maketitle

\abstract{\noindent This paper investigates linear programming based branch-and-bound
  using general disjunctions, also known as stabbing planes, for solving
  integer programs. We derive the first sub-exponential lower bound
  (in the encoding length $L$ of the integer program) for the size of a general
  branch-and-bound tree for a particular class of (compact) integer
  programs, namely $\smash{2^{\Omega(L^{1/12 -\epsilon})}}$ for every~$\epsilon >0$.
  This is achieved by showing that general branch-and-bound
  admits quasi-feasible monotone real interpolation, which allows us to
  utilize sub-exponential lower-bounds for monotone real circuits
  separating the so-called clique-coloring pair.
  Moreover, this also implies that refuting $\Theta(\log(n))$-CNFs requires
  size~$2^{n^{\Omega(1)}}$ branch-and-bound trees with high probability by considering the closely related notion of
  infeasibility certificates introduced by Hrube{\v{s}} and Pudl{\'a}k~\cite{hrubevs2017random}.
  One important ingredient of the proof of our interpolation result
  is that for every general branch-and-bound tree proving
  integer-freeness of a product~$P\times Q$ of two polytopes $P$ and $Q$,
  there exists a closely related branch-and-bound tree for showing
  integer-freeness of~$P$ or one showing integer-freeness of~$Q$. Moreover,
  we prove that monotone real circuits can perform binary search
  efficiently.
}
\section{Introduction}

In recent years, there has been renewed interest in the proof system
associated to branch-and-bound using general disjunctions for solving
integer linear programs (ILPs)~\cite{beame2017stabbing,fleming2021power,Basu2022,dantchev2022depth,dey2022lower,glaser2023computing}; the
literature sometimes also uses the name ``Stabbing Planes'' (SP),
see~\cite{beame2017stabbing}. In each node, a general disjunction of
the form $\alpha\T x \leq \delta \lor \alpha\T x \geq \delta +1$ for $\alpha \in \Z^n$, $\delta \in \Z$
is used to create two child nodes.

Branching on general disjunctions lies at the core of Lenstra's algorithm
for integer programming in fixed dimension~\cite{Len83}. It has also been
used, for example, for special ordered sets~\cite{BeaF76}, exploiting
flatness~\cite{DerV06}, achieving feasibility~\cite{MahC13}, and symmetry
handling~\cite{OstLRS11}.

Nevertheless, the dominant strategy in practice is to employ variable
branching of the form $x_i \leq \delta \lor x_i \geq \delta + 1$ for
some variable~$x_i$, which is the special case where $\alpha$ is the $i$th
unit vector. Some reasons for this choice are that the selection of a
branching disjunction is easier, the sparsity of the constraint matrix is
not increased, and it often allows to fix variables, e.g., if the variables
are binary. This variable branching strategy is then usually enhanced by
the application of cutting planes like Chv\'atal-Gomory cuts in a
branch-and-cut framework, which has seen a tremendous improvement over the
last decades.

As a proof system, branch-and-bound with general disjunctions is not only a
generalization of branch-and-bound using variable disjunctions, but also
branch-and-cut with variable disjunctions and Chv\'atal-Gomory
cuts, see Beame et al.~\cite{beame2017stabbing} (in fact it is even
equivalent to branch-and-cut with general disjunctions
and split-cuts).  Hence, lower bounds on the size of a
branch-and-bound tree using general disjunctions are also lower bounds on
the size of a branch-and-cut tree using Chv\'atal-Gomory cuts. This fact
shows that branch-and-bound using general disjunctions form a quite general
and important algorithm class.

It is thus surprising that so far no family of integer linear programs
provably requiring branch-and-bound trees using general disjunctions of
super-polynomial size (in the encoding length of the program)
without some kind of caveat is known. In fact, no super-linear bounds are
known. In this paper we close this gap by providing a class of compact
integer programs requiring branch-and-bound trees using general
disjunctions of size~$2^{\Omega(L^{1/12 -\epsilon})}$ for every~$\epsilon > 0$,
where $L$ denotes the encoding length of the ILP. This has been posed
as an open problem by Dadush and Tiwari~\cite{dadush2020complexity}.
\smallskip

We briefly survey previous contributions: It is actually relatively easy to
give families of ILPs which require branch-and-bound trees of
exponential size \emph{in the number of variables of the ILP} (but
not the encoding size of the ILP). Here the two main strategies are
the following: Dadush and Tiwari~\cite{dadush2020complexity} argued
that an ILP which is barely infeasible (i.e., removal of any
constraint makes the ILP feasible) must require large
branch-and-bound trees, since it is impossible to construct a
certified branch-and-bound tree which does not use every constraint in
at least one Farkas-certificate at its leaves.
More accurately, the obtained bound is the number of constraints divided
by the number of variables, which is only strong for a large number of constraints. This weakness is
mitigated by an extended formulation
of the ILP they use (with polynomial encoding size in the number
of variables).  However, this formulation also uses continuous variables.
Their strategy was later generalized by Dey et
al.~\cite{dey2022lower}. The other strategy, as investigated by Gl\"aser and
Pfetsch~\cite{glaser2023computing}, is based on finding a large set of points which
have to be associated to different leaves of some given
branch-and-bound trees. Formally, it considers hiding sets, which have
been introduced by Kaibel and Weltge~\cite{kaibel2015lower}.  However,
it seems impossible to derive bounds on the size of a
branch-and-bound tree which exceed the number of the constraints of
the ILP via either of these two strategies.

Despite the fact that no strong lower bounds on the size of a
branch-and-bound tree have been available prior to this paper, Beame et
al.~\cite{beame2017stabbing} gave a family of unsatisfiable CNF
formulas, such that refuting the corresponding ILP requires a
branch-and-bound tree of \emph{depth}~$\Omega(n/\log^2 n)$.  Note that
since trees are not necessarily balanced, this does not yield a good
lower bound on the size of a tree.

Besides lower bounds, there are some structural insights into branch-and-bound
using general disjunctions:
Dadush and Tiwari~\cite{dadush2020complexity} have shown that branch-and-bound
using general disjunctions does not become weaker (with respect
to polynomial simulation) when restricting the coefficients of the disjunctions
to have polynomial \emph{encoding length} (cf.\ Theorem~\ref{thm:coeff_size} below),
which is crucial for our argument. If we restrict the coefficients of branch-and-bound using
general disjunctions to polynomial \emph{size} then branch-and-bound can be quasi-polynomially
simulated by the Chv\'atal-Gomory cutting planes proof system (CG-CP), see Fleming et al.~\cite{fleming2021power}.
Thus, known lower bounds for
CG-CP~\cite{pudlak1997lower,hrubevs2017random,fleming2022} can be lifted to branch-and-bound using general disjunctions with polynomially bounded coefficients.

The strategy we employ in this paper is to show that branch-and-bound
using general disjunction admits \emph{quasi-feasible real monotone
  interpolation} and then lift lower bounds for monotone real circuits
separating the so-called clique-coloring pair
given by Pudl\'ak~\cite{pudlak1997lower}
(cf.\ Theorem~\ref{thm:mon_circ_lb_BMS}) to lower bounds for
branch-and-bound trees. This idea is explained in
Section~\ref{sec:prelim} and has already been successfully used for
other proof systems. Most prominently, \cite{pudlak1997lower}
derived sub-exponential lower bounds for CG-CP for the same ILP
used below by showing that CG-CP admits
real feasible monotone interpolation and a similar result for the
resolution proof system. Dash showed an analogous result for the
cutting plane proof system using lift-and-project cuts~\cite{dash2005exponential} and later
for split cuts~\cite{dash2010complexity}, which generalize both Chv\'atal-Gomory and lift-and-project cuts.
The concept of feasible interpolation and
the method in which it is used to derive lower bounds has been
developed in the sequence of
papers~\cite{krajicek1994,razborov1995unprovability,krajicek1997,BPR1997,pudlak1997lower}.

Note however, that feasible monotone interpolation can only be used to obtain lower bounds
for problems in a very specific form. This limitation has recently been addressed
independently in~\cite{hrubevs2017random} and~\cite{fleming2022}, where it is shown
that random $\Theta(\log(n))$-CNFs are hard for CG-CP.
To this end, \cite{hrubevs2017random} introduced the concept of
infeasibility certificates which are very closely related to
the notion of feasible interpolation. We mimic their approach
to lift real monotone circuit lower bounds for
infeasibility certificates for random CNFs to establish that $\Theta(\log(n))$-CNFs
require branch-and-bound trees with general disjunctions of size at least
$\smash{2^{n^{\Omega(1)}}}$ with high probability as well.

The rest of this paper is structured as follows:
We first survey some necessary preliminaries in Section~\ref{sec:prelim}.
Then we explicitly state our results in Section~\ref{sec:results}
and describe in which way they are related. The proofs are then given in
Section~\ref{sec:proofs}.
  
\section{Preliminaries}
\label{sec:prelim}

For polyhedra~$P \subseteq \R^{n_1}$ and~$Q \subseteq \R^{n_2}$, let~$P
\times Q = \{\binom{x}{y} \in \R^{n_1+n_2}\,|\, x \in P,\; y \in Q\}$ denote their Cartesian product.
For $n \in \N$, we use $[n] \define \{1, \dots, n\}$.

\paragraph{Systems of Linear Inequalities}
Let~$Ax\leq b$ with~$A\in \Q^{m\times n}$ and~$b\in \Q^m$ be a system of
linear inequalities.  By scaling, we can assume that~$A$ and~$b$ have
integral entries. If a polyhedron~$P$ is described by $Ax\leq b$ we
write $\{Ax\leq b\}\coloneqq\{x~|~Ax\leq b\} = P$ for brevity.

We say~$Ax\leq b$ is \emph{integer-feasible}, if there is a
point~$\hat x\in \Z^n$ with~$A \hat x\leq b$ and \emph{integer-infeasible}
otherwise. The polyhedron $\{Ax \leq b\}$ is \emph{integer-free} if
$Ax \leq b$ is integer-infeasible. Similarly, $Ax\leq b$ is
\emph{LP-feasible}, if there is a point~$\hat x\in \Q^n$
with~$A \hat x\leq b$ and \emph{LP-infeasible} otherwise.
  
A \emph{Farkas-certificate (of infeasibility)}
for the system $Ax \leq b$ is a vector~$f\in \Z^m_+$, such that~$f^\top A = 0$
and~$f^\top b <0$. It is well known that a linear system~$Ax\leq b$ is
LP-infeasible if and only if it admits a Farkas-certificate of
infeasibility. Note that the restriction to integral~$f$ is without
loss of generality. 
  
  \paragraph{Branch-and-Bound Trees}

  To fix notation, we formalize branch-and-bound trees.  A
  \emph{disjunction} is a pair of linear inequalities of the form
  $(\alpha^\top x\leq \delta,\; \alpha^\top x\geq \delta +1)$, where~$\alpha\in \Z^n$
  and~$\delta \in\Z$, which we denote
  $\alpha^\top x\leq \delta \lor \alpha^\top x\geq \delta +1$. Note that every integer
  point satisfies exactly one of them. A \emph{branch-and-bound proof
    (of integer-infeasibility)} or \emph{branch-and-bound tree} for an
  system of linear inequalities~$(P)$ is a rooted binary directed
  tree~$T$ with the following properties:
\begin{enumerate}
\item \label{prop:labeling} For every non-leaf node~$N$ there is a disjunction
  $\alpha^\top x\leq \delta \lor \alpha^\top x\geq \delta +1$, such that the \emph{left}
  outgoing edge of~$N$ is labeled with the inequality $\alpha^\top x\leq \delta$
  and the \emph{right} edge is labeled with $\alpha^\top x\geq \beta +1$. The
  neighbor~$N_\leq$ of~$N$ incident to the left edge is called the
  \emph{$(\alpha^\top x\leq \delta)$-child} of~$N$, whereas the neighbor~$N_\geq$
  incident to the right edge is the
  \emph{$(\alpha^\top x\geq \delta +1)$-child}.  The
  \emph{$(\alpha^\top x\leq \delta)$-branch} at a node~$N$ is the directed
  subtree~$T(N_\leq)$ rooted at~$N_\leq$ and the \emph{$(\alpha^\top x\geq \delta +1)$-branch}
  is the subtree~$T(N_\geq)$ rooted at $N_\geq$.

\item \label{prop:validity} For a node~$N$ in~$T$, the
  problem~$T_N(P)$ \emph{associated to~$N$ in a branch-and-bound
    tree~$T$ for~$(P)$} is (the LP-relaxation) of~$(P)$ and all constraints
  occurring as edge labels on the unique path from the root to~$N$
  in~$T$.
  We say~$N$
  is \emph{feasible} if~$T_N(P)$ is LP-feasible and \emph{infeasible}
  otherwise. We require~$N$ to be infeasible for every leaf~$N$ of~$T$.
\end{enumerate}
A \emph{branch-and-bound tree for an integer-free
  polyhedron~$P$} is
a branch-and-bound tree for a system~$Ax\leq b$ with~$P = \{Ax\leq b\}$.

We emphasize that we do not require non-leaf nodes of~$T$ to be feasible.
This is convenient, since then a branch-and-bound tree~$T$ for a polyhedron
$P$ is also a branch-and-bound tree for every polyhedron~$Q\subseteq P$,
see, e.g., Dey et al.~\cite{dey2022lower}.
Moreover, this does not alter the minimal size of a
branch-and-bound tree for any infeasible integer problem.

A tree labeled as described in~Property~\ref{prop:labeling}
which does not necessarily satisfy
Property~\ref{prop:validity} will be called \emph{not necessarily valid
  branch-and-bound tree}. For emphasis, we sometimes call trees that
satisfy both properties \emph{valid}. Note that with above definition
only integer-infeasible systems of linear inequalities have
valid branch-and-bound trees.

Since we consider only binary trees, the number of nodes of a tree will be asymptotically twice the number of its leaves. Hence, we may define the \emph{size~$\card T$ of~$T$} to denote the number of leaves of~$T$, which turns out to be slightly more convenient.
 We let~$\SBB(P)$ be the smallest size of a branch-and-bound tree using general disjunctions proving integer-freeness of~$P$.
For~$P$ containing integral points, we define~$\SBB(P) \coloneqq +\infty$.

An important result about branch-and-bound trees is that they can be
recompiled to reduce the encoding length of the coefficients used in
the disjunctions at the nodes.
This is stated in the following Theorem by Dadush and Tiwari~\cite{dadush2020complexity}:

\begin{theorem}{(Theorem 1 in~\cite{dadush2020complexity} and its proof)}\label{thm:coeff_size}
  Let~$P \subseteq \R^n$  be an integer-free polytope contained in the ball~$B^n_1(R) \coloneqq \{x\in \R^n ~|~ \lVert x \rVert_1 \leq R\}$ with radius~$R \in \N$ with respect to the~$\ell_1$-norm. Let~$T$ be a branch-and-bound tree showing the integer-freeness of~$P$. Then, there exists a branch-and-bound tree~$T'$ for~$P$, such that~$\card{T'} \leq (4n+5) \card  T$, and for every disjunction~$\alpha^\top x \leq \delta \lor \alpha^\top x \geq \delta+1$ which~$T'$ branches on
  we have
  $\max \{\lVert \alpha \rVert_\infty,\card{\delta}\} \leq (10nR)^{(n+2)^2}$.
\end{theorem}

A \emph{certified branch-and-bound tree (for a system of
inequalities/polyhedron)}~$T$ is a branch-and-bound tree (for a system of
inequalities/polyhedron), where attached to
every leaf~$L$ is a Farkas-certificate~$f^L$ of infeasibility for the
problem associated to~$L$.

  \paragraph{Monotone Real Circuits}

  A \emph{monotone real circuit}~$C$ is an acyclic directed graph whose vertices are called \emph{gates}, such that every gate has either zero or two incoming edges and the incoming edges at every gate are ordered. The number of incoming edges of a gate is called its \emph{fan-in}. If gate~$g$ has fan-in zero, it is called an \emph{input gate} and is labeled with a variable~$x_i$ and if~$g$ has fan-in two, then~$g$ is labeled by a non-decreasing function~$f_g\colon\R^2\rightarrow \R$, which is the \emph{function applied at~$g$}. If~$x= (x_1,\dots,x_k)$ are the variables occurring as labels of input gates, we define (slightly abusing notation) the function~$g\colon \R^k \rightarrow \R$ \emph{computed by~$g$} inductively (along a topological order of the underlying graph) to be~$g(x) = x_j$, if~$g$ is an input gate labeled by~$x_j$ and
  by~$f_g(g_1(x),g_2(x))$, if~$g$ is a non-input gate and~$g_1$ and~$g_2$ are the first and second predecessors of~$g$. Finally, there is a designated output gate~$h$, and the value~$C(x)$ computed by~$C$ on input~$x$ is~$h(x)$. The \emph{size}~$\card C$ of a circuit~$C$ is the number of its gates.

  Note that bounded fan-in is essential, since monotone real circuits
  with unbounded fan-in of linear size can compute arbitrary monotone
  functions (consider the circuit with a unique non-input gate
  connected to all inputs). Moreover, the term `real monotone'
  circuit is slightly misleading: the arithmetical structure of the
  real numbers $\R$ does not play a role in above definition -- instead
  $\R$ merely plays the role of a sufficiently large linearly ordered
  domain (this has already been mentioned in~\cite{pudlak1997lower}).

  A circuit~$C$ \emph{decides} (the membership problem for) a
  set~$X \subseteq \R^k$, if~$C(x) = 1$ for~$x\in X$ and~$C(x) = 0$
  otherwise.  Similarly, $C$ separates two sets~$Z_1\subseteq \R^k$
  and~$Z_2\subseteq \R^k$, if~$C(z) = 1$ for all~$z\in Z_1$
  and~$C(z) = 0$ for all~$z\in Z_2$ or vice versa.  While modifying a
  given circuit, \emph{post-composing} the
  function~$f_g$ applied at a gate~$g$ with a
  function~$\varphi\colon \R \rightarrow \R$ means replacing~$f_g$
  by~$\varphi \circ f_g$ , where~$\circ$ denotes
  the composition of functions. Similarly, \emph{pre-composing the first [second] input}
  means replacing the function $f_g\colon x,y\mapsto f_g(x,y)$ by
  $x,y \mapsto f_g(\varphi(x),y)$ [${x,y \mapsto f_g(x,\varphi(y))}$].
  The notion of monotone real circuits was introduced in~\cite{pudlak1997lower}.

\paragraph{Interpolation}

We briefly translate the notion of interpolation into the language of linear integer
programs. For this, we consider integer-infeasible linear inequality systems of the following form:
\begin{equation}\label{eq_interpol_templ}
    \begin{pmatrix}A\\0\end{pmatrix}x +
    \begin{pmatrix}0\\B\end{pmatrix}y +
    \begin{pmatrix}C\\D\end{pmatrix}z \leq \begin{pmatrix}a\\b\end{pmatrix},\qquad
    x,\, y,\, z \in \{0,1\}^{n_1+n_2+n_3},
\end{equation}
where $C \geq 0$ and~$D\leq 0$ (entry-wise). Moreover, we let
$A \in \Z^{m_1 \times n_1}$, $B \in \Z^{m_2 \times n_2}$,
$C \in \Z^{m_1 \times n_3}$, $D \in \Z^{m_2 \times n_3}$,
$a \in \Z^{m_1}$, $b \in \Z^{m_2}$ and~$n \define n_1+n_2+n_3$.

Since \eqref{eq_interpol_templ} is integer-infeasible, at least one of the
systems~$Ax \leq a-Cz,\; x\in \{0,1\}^{n_1}$ or
$By \leq b-Dz,\; y\in \{0,1\}^{n_2}$ is infeasible for every fixed
$z \in \{0,1\}^{n_3}$. An \emph{interpolant} for~\eqref{eq_interpol_templ}
is a binary function~$I\colon \{0,1\}^{n_3} \rightarrow \{0,1\}$, such
that~$I(z) =1$ implies that the first system is infeasible and~$I(z) = 0$
implies that the second is infeasible.
Note that there is a interpolant for every integer-infeasible system of the form~\eqref{eq_interpol_templ}. We say a proof system~$\S$
\emph{admits (monotone/real monotone) [quasi-]feasible interpolation}, if
for any proof of infeasibility of~\eqref{eq_interpol_templ} in~$\S$ of
size~$S$ there exists a (monotone/real monotone) circuit of size
[quasi-]polynomial in~$S$ and the encoding size
of~(\ref{eq_interpol_templ}) computing such an interpolant.  Then, if every
such monotone circuit must be large, also any proof
for~(\ref{eq_interpol_templ}) in~$\S$ must be large. This is the case for the
examples described in the next paragraph.

\paragraph{The Clique-Coloring Pair and the Broken-Mosquito-Screen Pair}

The following construction is based on the observation that no $r$-vertex graph $G$ simultaneously admits both a $(k-1)$-coloring and a $k$-clique. This is expressed by the integer-infeasibility of:
\begin{subequations}\label{eq:BMS_alternative}
\begin{align}
  x_{i\ell}+x_{j\ell} &\leq 2 - z_{ij}&&\forall\, \{i,j\} \in \binom{[r]}{2},\; \forall\, \ell\in[k-1], \\
  \sum_{\ell\in[k-1]}x_{i\ell} &\geq 1 && \forall\, i\in[r], \\
  y_{i}+y_{j} &\leq 1 + z_{ij} && \forall\, \{i,j\} \in \binom{[r]}{2}, \\
  \sum_{i\in[r]} y_{i} &\geq k, &&  \\
  x \in [0,1]^{r \times (k-1)},\; & y \in [0,1]^{r},\; z \in [0,1]^{{[r]\choose 2}}, &&
\end{align}
\end{subequations}
where we interpret~$z\in \{0,1\}^{[r] \choose 2}$ as an encoding of the graph~$G = ([r],E)$, with~$z_{ij} = 1$ if and only if~$\{i,j\}\in E$, $x \in \{0,1\}^{r \times (k-1)}$ as a $(k-1)$-coloring of $G$, with $x_{i\ell} = 1$ if and only if vertex~$i$ gets assigned color $\ell$, and~$y\in \{0,1\}^{r}$ as an~$k$-clique of~$G$, with $y_{i} =1$ if and only the clique contains vertex $i$.
  
If we write~\eqref{eq:BMS_alternative} in the form of~\eqref{eq_interpol_templ}, the pair of sets
\begin{align*}
  & Z_1 \define \{z \in \{0,1\}^{[r]\choose 2}~|~ \exists x \in \{0,1\}^{r\times (k-1)}\colon Ax \leq a- Cz\} \text{ and }\\
  & Z_2 \coloneqq\{z \in \{0,1\}^{[r]\choose 2}~|~ \exists y \in \{0,1\}^{r}\colon By \leq b- Dz\}
\end{align*}
is known as the \emph{clique-coloring pair} or \emph{CC-pair}.
Note that $n_1 = r(k-1)$, $n_2 = r$, and $n_3 = {r \choose 2} = (r^2-r)/2$ for later calculations.
As already remarked, we have $Z_1 \cap Z_2 = \emptyset$.

Pudl\'ak~\cite{pudlak1997lower} gave the following lower-bound for any monotone real circuit  separating the CC-pair:

\begin{theorem}\label{thm:mon_circ_lb_BMS}
  Every family of monotone real circuits separating the CC-pair with $r$ vertices and $k \define \lfloor \frac 1 8 (r /\log r )^{2/3}\rfloor$ has size~$2^{\Omega((r/\log r )^{1/3})}$.
\end{theorem}

Note that $\log$ always refers to the logarithm with respect to base $2$ in
this paper.

Since $n_3 = (r^2-r)/2$ is the number of inputs to a circuit separating the CC-pair, we note
\[
  2^{\Omega((r/\log r )^{1/3})} = 2^{\Omega((n_3/\log n_3)^{1/6})} \in 2^{\Omega( n_3^{1/6 -\epsilon})} \text{ for every }\epsilon >0.
\]

Theorem~\ref{thm:mon_circ_lb_BMS} is in contrast to the fact that the CC-Pair can be separated in polynomial time by semi-definite programming using Lov\'asz' Theta body (cf., e.g., Remark 9.3.20(b) in~\cite{grotschel1988geometric}).

Another example of a disjoint pair of languages which requires a large monotone real circuit to be separated is the \emph{broken-mosquito-screen pair} or \emph{BMS-pair}.
The BMS-pair is polynomially equivalent to the CC-pair~\cite{pudlak2003reducibility} and is therefore also polynomial time separable. For the BMS-pair, a slightly more explicit bound is available in the literature:
Cook and Haken~\cite{haken1999exponential} show that any circuit separating the BMS-pair has at least~\smash{$2^{K n_3^{1/8}}$} gates for some~$K> 0.32$.

\paragraph{Random CNFs and Infeasibility Certificates}
In this section we will consider Boolean variables $x_1,\dots,x_n$.
A \emph{literal} is a variable $x_i$ or its negation $\lnot x_i$. A \emph{clause} is a
set of literals. A \emph{CNF} (or formula in \emph{conjunctive normal
  form}) is a set of clauses. A \emph{$k$-clause} is a clause with $k$ variables and
a \emph{$k$-CNF} is a CNF containing only $k$-clauses. The \emph{satisfiability problem} is the
problem of deciding whether there exists a satisfying assignment to a
given CNF $\C$, i.e., an assignment $\alpha$ of binary values (or truth values)
to $x_1,\dots,x_n$, such that every clause $C\in\C$ contains either 
a literal $x_i$ for a variable which is assigned to be true (i.e., $\alpha(x_i) =1$)
or a literal $\lnot x_i$ for a variable which is assigned to be false (i.e., $\alpha(x_i) =0$).

The satisfiability problem can be cast as an ILP in a straight-forward way:
\begin{equation}\label{eq:CNF_ILP}
   \sum_{x_i\in C} x_i +\sum_{\lnot x_i\in C}(1-x_i) \geq 1    \quad\forall C \in \C,\qquad
   x_1,\dots,x_n \in \{0,1\}.
 \end{equation}
Thus we can speak about branch-and-bound trees refuting unsatisfiable CNFs.

A random $k$-CNF $\C$ in $n$ variables and $m$ clauses is obtained by
picking $k$-clauses uniformly and independently at random from the
${n \choose k}2^k$ possible $k$-clauses. Since we allow repetition,
$\C$ may contain less than $m$ clauses.  We are interested in choosing
the parameters $m$ and $k$ in a way, such that the resulting CNFs are
unsatisfiable, but hard to refute with high probability.  Moreover, we
are interested in choosing these parameters as small as possible.  For
this we follow the choices made in~\cite{hrubevs2017random}, where this is
discussed in more detail.  First observe
that any random assignment to the variables satisfies $\C$ with
probability $(1-2^{-k})^m$. Hence, by the union bound there is a
satisfying assignment with probability at
most~$\smash{(1-2^{-k})^m2^n\leq e^{-2^{-k}m}2^n}$. Hence, if we choose
$m \geq (\ln 2 +\varepsilon)2^kn$ for some $\varepsilon > 0$, then a
random-formula is unsatisfiable with high probability, where $\ln$
refers to the natural logarithm. Note that if
$k \in O(\log(n))$, then the number of clauses is polynomial in~$n$.

Let $\C$ denote a CNF and $X_0 \cup X_1 = X \coloneqq \{x_1,\dots,x_n\}$ denote a partition
of its variables.
Every clause $C_i$ in $\C$ can be written as as $C_i = C_i^0\cup C_i^1$, where $C_i^j$ contains
only variables from $X_j$ and their negations.
A monotone Boolean function $F:\{0,1\}^m \rightarrow \{0,1\}$ is an~\emph{{$(X_0,X_1)$}-certificate (of infeasibility)} for $\C$,
if for every $A \subseteq [m]$ we have
\begin{align*}
  F(A) = 0 &\implies \text{$\{C_i^1\colon i\in [m]\setminus A\}$ is unsatisfiable},\\
  F(A) = 1 &\implies \text{$\{C_i^0\colon i\in A\}$ is unsatisfiable},
\end{align*}
where we identify a set of (indices of) clauses with its characteristic vector.
That is, given a subset $A$ of the clauses of $\C$ as input, $F$
determines one of the CNFs $\{C_i^1\colon i\in [m]\setminus A\}$ or
$\{C_i^0\colon i\in A\}$ which is unsatisfiable. It is easy to see
that, for any choice of $X_0$ and $X_1$, $\C$ admits an $(X_0,
X_1)$-certificate if and only if it is unsatisfiable (Proposition~5
in~\cite{hrubevs2017random}).

Similarly to interpolants, infeasibility certificates have high
monotone real circuit complexity for some problems, among them the
interesting case of random CNFs. We say that a sequence of
events~$(E(n))_{n\in \N}$
 \emph{holds with high probability} if
$\lim_{n \rightarrow \infty} \Prob[E(n)] = 1$.
We then have:

\begin{theorem}[Theorem~2 in~\cite{hrubevs2017random}]\label{thm:CNF_cert_compl}
  Let $c > 1$ be a constant and let $n \geq 1$ be given. Let
  $X_0 \cup X_1$ be a partition of $2n$ variables into two sets of
  equal size. If $\C$ is a random $k$-CNF with $O(n2^k)$ clauses,
  variables $X_0 \cup X_1$, and $k \geq c \log (n)$, then every
  $(X_0,X_1)$-certificate for $\C$ requires monotone real circuits of size
  $\smash{2^{n^{\Omega(1)}}}$ with high probability.
\end{theorem}

Note that the conditions in Theorem~\ref{thm:CNF_cert_compl} do not
ensure that~$\C$ is unsatisfiable with high probability.  Hence, for
small values of $c$, the theorem states that with high
probability either every $(X_0,X_1)$-certificate
requires large real monotone circuits or~$\C$ is satisfiable (and thus there are no
$(X_0,X_1)$-certificates at all). However, for $c > \ln 2$ almost all random CNFs
as in the statement of the theorem are unsatisfiable, cf.\ the discussion above.
A similar remark is true about
most results about random CNFs in this paper.


The connection between infeasibility certificates and interpolation
is explained via the observation that
an infeasibility certificate for a CNF $\C = \{C_1,\dots,C_m\}$ can be seen as an interpolant to a
closely related CNF: Let $X_0\cup X_1$ be a partition of the variables in $\C$. We introduce additional
variables $ Y = \{y_1,\dots,y_m\}$, one for every clause, and consider
the CNF $\D$ with clauses
\[
  C_1^0\cup\{\lnot y_1\},\; \dots,\; C_m^0\cup\{\lnot y_m\},\;
  C_1^1\cup\{y_1\},\; \dots,\; C_m^0\cup\{y_m\}.
\]

Let $\D_0$ denote the CNF containing the first $m$ clauses (with variables $X_0\cup Y$)
and $\D_1$ the CNF containing the second $m$ clauses (with variables $X_1\cup Y$).
Let $Y_i$ (where $i\in\{0,1\}$) denote the collection of assignments $\alpha$ to the variables in~$Y$ which
make $\D_i$ satisfiable, if we fix a variable $y\in Y$ to $\alpha(y)$.
It is then easy to see that any interpolant separating $Y_0$ and $Y_1$
is an $(X_0,X_1)$-certificate for $\C$ and vice versa.
In particular, this observation implies that interpolation theorems
also convert short proofs into infeasibility certificates with low
monotone circuit complexity, which can be used to lift the lower bound
given by Theorem~\ref{thm:CNF_cert_compl} to a lower bound on the size of proofs,
provided $\D$ is not significantly harder to refute than $\C$ (which
typically does not seem to be the case).

\section{Results}
\label{sec:results}

The central result of this work is that branch-and-bound using general disjunctions admits quasi-feasible monotone real interpolation,
that is:

\begin{theorem}\label{thm:qfeas_real_monotone_interpolation}
  Given a branch-and-bound tree~$T$ for~\eqref{eq_interpol_templ}, there exists a monotone real circuit of size
  $50(n+1)^2\card{T}^2 \cdot [(n+2)^2\log(10n^3+3)]^{\log((4n+5)\card T)}$
 with input~$z$,
  which separates the sets $Z_1 \coloneqq \{z \in \{0,1\}^{n_3}~|~ \exists x\in \{0,1\}^{n_1}\colon Ax \leq a- Cz\}$ and $Z_2 \coloneqq \{z \in \{0,1\}^{n_3}~|~ \exists y \in \{0,1\}^{n_2}\colon By \leq b- Dz\}$.
\end{theorem}

Choosing~\eqref{eq_interpol_templ} to be integer linear programs expressing a separation problem for which we have lower bounds for separating monotone real circuits,
we obtain lower bounds for branch-and-bound trees for~\eqref{eq_interpol_templ}.
For example, combining Theorems~\ref{thm:qfeas_real_monotone_interpolation} and~\ref{thm:mon_circ_lb_BMS}, we immediately obtain the following sub-exponential bound:

\begin{theorem}\label{thm:bb_lb_BMS}
  Every family of branch-and-bound trees for~\eqref{eq:BMS_alternative}, where $k \define \lfloor \frac 1 8 (r /\log r )^{2/3}\rfloor$, has size at least~$2^{\Omega(n^{1/6 -\epsilon})}$, for every~$\epsilon >0$, where $n$ is the number of variables.
\end{theorem}

We note that one could make this bound completely explicit, i.e., for fixed~$\epsilon>0$, we can give~$N\in \N$ and~$\delta >0$, such that any branch-and-bound tree for~\eqref{eq:BMS_alternative} has size at least~$2^{\delta n^{1/6 -\epsilon}}$ for every~$n\geq N$, by tracking the factor hidden in the $\Omega$-notation in Theorem~\ref{thm:mon_circ_lb_BMS} through its proof (see~\cite{pudlak1997lower} building on~\cite{alon1987monotone} as presented in~\cite{wegener1987complexity}).
Alternatively, we can give a similar bound for the BMS-pair, which can then
be made explicit by some elementary calculations.

Since the encoding length~$L$ of~\eqref{eq:BMS_alternative} satisfies $L \in \Theta(n^2)$, we immediately obtain: 
\begin{cor}
  Every family of branch-and-bound trees for~\eqref{eq:BMS_alternative},
  where $k \define \lfloor \frac 1 8 (r /\log r )^{2/3}\rfloor$, has size
  at least~$2^{\Omega(L^{1/12 -\epsilon})}$, for every~$\epsilon >0$, where
  $L$ is the encoding length~\eqref{eq:BMS_alternative}.
\end{cor}

We also obtain a similar result for random CNFs:

\begin{theorem}
  \label{thm:3CNF_hard}
  If $\C$ is a random $k$-CNF with $2n$ variables and
  $O(n\,2^k)$-clauses, where $k\geq c\log n$ for a constant $c>1$, then any branch-and-bound
  tree for~(\ref{eq:CNF_ILP}) for $\C$
  has size at least~$2^{n^{\Omega(1)}}$ with high probability.
\end{theorem}

\smallskip

The remainder of this section outlines how we prove Theorem~\ref{thm:qfeas_real_monotone_interpolation}.
For this we will require three ingredients.

Let~$P\in \R^{n_1}$ and~$Q\in \R^{n_2}$ be two polytopes.  Our first
ingredient is the fact that given a certified branch-and-bound
tree~$T$ showing the integer-freeness of~$P \times Q$, there is
a branch-and-bound tree for the integer-freeness of~$P$, which
is structurally very close to~$T$, or there is such a tree showing the
integer-freeness of~$Q$. This is helpful for showing
Theorem~\ref{thm:qfeas_real_monotone_interpolation}, since after
fixing the variables~$z$, the feasible region of the LP-relaxation
of~(\ref{eq_interpol_templ}) is a product of two lower-dimensional polytopes:
\[
  \{x\in[0,1]^{n_1}~|~Ax \leq a- Cz\} \times \{y\in[0,1]^{n_2}~|~
  By \leq b- Dz\}.
\]

Given a branch-and-bound tree~$T$ for showing integer-freeness of
$P\times Q \subseteq \R^{n_1 \times n_2}$, a branch-and-bound tree~$T'$ for
$P$ \emph{conforms to~$T$}, if their underlying directed graphs
(including the ordering of the children as~$\leq$- and
$\geq$-children) are identical and, if the disjunction used by~$T$ at
a node~$N$ is
$\alpha^\top x + \beta^\top y \leq \delta \lor \alpha^\top x + \beta^\top y \geq
\delta+1$, then the disjunction used by~$T'$ at~$N$ is
$\alpha^\top x \leq \delta' \lor \alpha^\top x \geq \delta'+1$ for some
$\delta' \in \Z$.

A precursor to our first ingredient will be the following Lemma.

\begin{lemma}{(Structural Interpolation Lemma)}\label{lem:conformal_interpolation}
  For every branch-and-bound tree~$T$ for an integer-free product of polytopes~$P\times Q$
  there exits
  \begin{enumerate}[label=(\alph*)]
  \item \label{case:confP} a branch-and-bound tree~$T^P$ for~$P$ conforming to~$T$ or
   \item \label{case:confQ} a branch-and-bound tree~$T^Q$ for~$Q$ conforming to~$T$.
  \end{enumerate}
\end{lemma}

By applying Lemma~\ref{lem:conformal_interpolation} to a smallest branch-and-bound tree
for~$P\times Q$, we immediately obtain the following result, which is of independent interest.

\begin{cor}
  $\SBB(P \times Q) = \min(\SBB(P), \SBB(Q))$
\end{cor}

Note that Lemma~\ref{lem:conformal_interpolation} does not give any
information of what the right hand side of the disjunctions used at
the nodes of~$T^P$ or~$T^Q$ should be, while almost every other
property of the tree is preserved.  The remaining two ingredients for
the proof of Theorem~\ref{thm:qfeas_real_monotone_interpolation} are
intended to supply this information.  To this end, it would be
helpful, if Farkas-certificates maintain their validity when
passing from~$T$ to~$T^P$ or~$T^Q$. This is due to the fact that it is
not clear how we can decide the LP-feasibility of a system of linear
inequalities (even without integrality constraints), since we
need to perform computations via monotone real circuits. However, it
is easy to check whether a given Farkas-certificate is valid.

Unfortunately, Lemma~\ref{lem:conformal_interpolation} does not seem to hold
for \emph{certified} branch-and-bound trees.
Indeed, the naive way to obtain Farkas-certificates for~$T^P$
(or~$T^Q$) from Farkas-certificates for~$T$ does not work, i.e., for a leaf~$L$ in~$T^P$ using the
projection of the Farkas-certificate~$f^L$ attached to~$L$ in~$T$ onto
constraints of $P$ and branching constraints.  As a counter example, consider the square of
the two-dimensional cross-polytope
\begin{align*}
  C_x^2 \times C_y^2 \define\quad & \{x_1,x_2\in[0,1] \colon\; \sum_{i\in S} x_i + \sum_{i\not\in S} (1-x_i) \leq \tfrac{3}{2}\quad \forall S\subseteq \{1,2\}\} \\
  \times &\{y_1,y_2\in[0,1] \colon\; \sum_{i\in S} y_i + \sum_{i\not\in S} (1-y_i) \leq \tfrac{3}{2}\quad \forall S \subseteq \{1,2\}\}
\end{align*}
and the branch-and-bound tree shown in Figure~\ref{fig:extree}
for~$C_x^2 \times C_y^2$. Here, let all Farkas-certificates at the leaves
have the form
\begin{equation*}
  \begin{aligned} 
  &&1\,\cdot &\qquad\textstyle(\sum_{i\in S} x_i + \sum_{i\not\in S} (1-x_i) \leq \tfrac{3}{2}) \\ \textstyle
  +&\textstyle\sum_{i\in S}&1\,\cdot &\qquad (-x_i \leq -1)\\\textstyle
  +&\textstyle\sum_{i\not\in S}&1\,\cdot &\qquad (x_i \leq 0)\\
  \midrule
  &&&\qquad (0 \leq -0.5)
\end{aligned}
\end{equation*}
for some~$S\subseteq \{1,2\}$ or the analogous form for variables
from~$C_y^2$. For any leaf, there is only one such choice.

It is not hard to see that no choice of new right-hand-sides~$\hat{\delta}$ can
make every Farkas-certificate~$\hat{f}^L$ obtained as described above valid for all
leaves simultaneously, when we try to obtain a conforming branch-and-bound
tree for either~$C_x^2$ or~$C_y^2$: For~$C_x^2$, $\hat{\delta}_r \leq -1$ has to hold for the
right-hand-side~$\hat{\delta}_r$ used in the disjunction at the root $r$, if all
Farkas-certificates in the~$\leq$-branch at the root are to be valid. But
then it is impossible for both Farkas-certificates at leaves in
the~$\geq$-branch to be valid simultaneously. The situation for~$C_y^2$ is
similar.

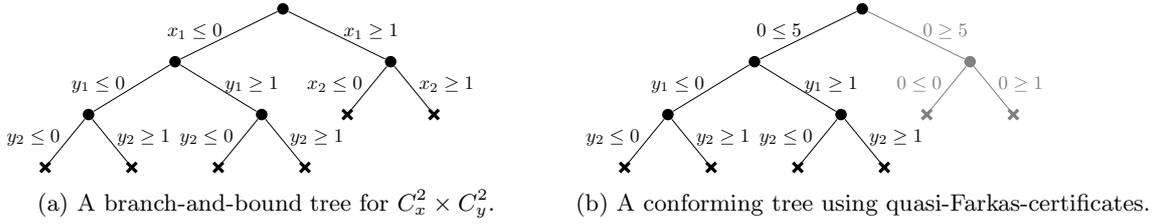
\begin{figure}
  \centering
  \begin{subfigure}[top]{0.47\linewidth}
    \scalebox{0.72}{
    \begin{forest}
for tree={circle, fill, inner sep=0pt, minimum size=6pt, l sep=20pt, s sep = 40pt}
[, 
    [, edge label={node[midway,left,yshift =2pt] {$x_1\leq 0$}}
      [,edge label={node[midway,left,yshift =2pt] {$y_1\leq 0$}},
        [,cross, edge label={node[midway,left,yshift =2pt] {$y_2\leq 0$}}]
        [,cross, edge label={node[midway,right,yshift =2pt] {$y_2\geq 1$}}]
      ]  
      [,edge label={node[midway,right,yshift =2pt] {$y_1\geq 1$}}
        [,cross, edge label={node[midway,left,yshift =2pt] {$y_2\leq 0$}}]
        [,cross, edge label={node[midway,right,yshift =2pt] {$y_2\geq 1$}}]
      ]
    ]
    [,edge label={node[midway,right,yshift =2pt] {$x_1\geq 1$}}
      [, cross, edge label={node[midway,left,yshift =2pt] {$x_2\leq 0$}}]
      [, cross, edge label={node[midway,right,yshift =2pt] {$x_2\geq 1$}}]
  ] 
]
\end{forest}}
\caption{A branch-and-bound tree for~$C_x^2 \times C_y^2$.}
  \label{fig:extree}
\end{subfigure}
\quad
\begin{subfigure}[top]{0.5\linewidth}
  \scalebox{0.72}{
    \begin{forest}
for tree={circle, fill, inner sep=0pt, minimum size=6pt, l sep=20pt, s sep = 40pt}
[, 
    [, edge label={node[midway,left,yshift =2pt] {$0\leq 5$}}
      [,edge label={node[midway,left,yshift =2pt] {$y_1\leq 0$}},
        [,cross, edge label={node[midway,left,yshift =2pt] {$y_2\leq 0$}}]
        [,cross, edge label={node[midway,right,yshift =2pt] {$y_2\geq 1$}}]
      ]  
      [,edge label={node[midway,right,yshift =2pt] {$y_1\geq 1$}}
        [,cross, edge label={node[midway,left,yshift =2pt] {$y_2\leq 0$}}]
        [,cross, edge label={node[midway,right,yshift =2pt] {$y_2\geq 1$}}]
      ]
    ]
    [,gray,edge label={node[midway,right,yshift =2pt] {$0\geq 5$}}, edge =gray
      [, cross, gray, edge label={node[midway,left,yshift =2pt] {$0\leq 0$}}, edge =gray]
      [, cross, gray, edge label={node[midway,right,yshift =2pt] {$0\geq 1$}}, edge =gray]
  ] 
]
\end{forest}}
\caption{A conforming tree using quasi-Farkas-certificates.}
  \label{fig:exconftree}
\end{subfigure}
\caption{A branch-and-bound tree for~$C_x^2 \times C_y^2$ showing that
  Lemma~\ref{lem:conformal_interpolation} does not hold for certified
  branch-and-bound trees. (b) shows a tree for~$C_y^2$ conforming to the one in (a) using
  quasi-Farkas-certificates. For the gray part of the tree, the second case
  from the definition of quasi-Farkas-certificates holds, i.e., we do not
  require the validity of Farkas-certificates there.}
\end{figure}

However, if we relax the notion of a Farkas-certificate very slightly,
then Lemma~\ref{lem:conformal_interpolation} holds also for certified
trees: Let~$P$ be a polytope and~$U$ an arbitrary set
with~$P \subseteq U$. A \emph{quasi-certified branch-and-bound
  tree}~$T$ for~$P$ relative to $U$ is a branch-and-bound tree
for~$P$, such that to every leaf~$L$ there is an attached
\emph{quasi-Farkas-certificate~$f^L$ relative to~$U$}, i.e., a
vector~$f^L \in \Z^{m_L}_+$ indexed by the~$m_L$ constraints of the
problem~$T_{L}(P)$ associated to~$L$, such that
  \begin{enumerate}
  \item $f^L$ is a valid Farkas-certificate for~$T_{L}(P)$ or
  \item an edge on the unique root-leaf path in~$T$ to~$L$ is labeled with a
    constraint~$\alpha^\top x \leq \gamma$ ($\alpha^\top x \geq \gamma +
    1$), such that~$U \cap \{\alpha^\top x \leq \gamma\} = \emptyset$ ($U
    \cap \{\alpha^\top x \geq \gamma + 1\} = \emptyset$).
  \end{enumerate}
  Note that~$f^L$ does not appear in the second condition.  A
  quasi-certified branch-and-bound tree~$T'$ for~$P$ \emph{conforms}
  to a quasi-certified branch-and-bound tree~$T$ for~$P\times Q$, if
  it does so as (uncertified) branch-and-bound tree and moreover the
  quasi-Farkas-certificate~$(f')^L$ attached to a leaf~$L$ of~$T'$ is
  the projection of the quasi-Farkas-certificate~$f^L$ attached to the
  leaf~$L$ in~$T$ onto variables corresponding to constraints of $P$
  and branching constraints. 

  In the above example, there exists a quasi-certified branch-and-bound relative to~$[0,1]^2$
  conforming to the previously problematic tree, which is shown in Figure~(\ref{fig:exconftree}).
  More generally, we have:

\begin{lemma}{(Certified Structural Interpolation Lemma)}\label{lem:conformal_cert_interpolation}
  For every quasi-certified branch-and-bound tree~$T$ for an
  integer-free product of polytopes~$P\times Q$ relative
  to~$U = U_P\times U_Q$, where~$P\subseteq U_P$,
  $Q \subseteq U_Q$ and $U_P$ and $U_Q$ are bounded,
  there exists
  \begin{enumerate}[label=(\alph*)]
  \item \label{case:cert_confP} a quasi-certified branch-and-bound tree~$T^P$ for~$P$ relative to~$U_P$ conforming to~$T$ or
   \item \label{case:cert_confQ} a quasi-certified branch-and-bound tree~$T^Q$ for~$Q$ relative to~$U_Q$ conforming to~$T$.
  \end{enumerate}
\end{lemma}

Passing from certified trees to quasi-certified trees allows us to
choose not to branch at a node~$N$ of~$T$ during construction
of~$T^P$
or~$T^Q$ and instead immediately proceed as in one of the subtrees
rooted at the children
of~$N$.  Conceptually, it might seem cleaner to consider certified
branch-and-bound trees which embed
into~$T$ for some suitable notion of embedding. However, this does not
appear to work well with certificates. Moreover, we do not want to be
forced to enumerate over all possible embedded trees later.
\smallskip

As previously mentioned, the last two ingredients for the proof of
Theorem~\ref{thm:qfeas_real_monotone_interpolation} instruct us how to
reconstruct the right-hand-sides of the disjunctions in the tree~$T^P$
or~$T^Q$ given by Lemma~\ref{lem:conformal_cert_interpolation}.  The
first ingredient is Theorem~\ref{thm:coeff_size}, which establishes
that the set of possible right-hand-sides is not too large (if we
recompile~$T$ beforehand).

It then remains to show that if a tree~$T^P$ as in
Lemma~\ref{lem:conformal_cert_interpolation} exits, we can search
the space of possible right-hand-side -- whose size is limited by
Theorem~\ref{thm:coeff_size} -- efficiently via binary search, even in
the quite restricted computational model of monotone real circuits.


More concretely, assume we have chosen values for the right-hand-sides
of all disjunctions used at a descendant of a node $N$ in $T$ and are now
choosing a value~$\gamma_N$ for the disjunction used at~$N$.
If we consider (say) the $(\alpha_N^\top x\geq \gamma_N+1)$-branch at
$N$, then we want to choose $\gamma_N$ as small as possible (i.e., we
want to choose the inequality $\alpha_N^\top x\geq \gamma_N+1$ as weak as
possible), such that we still obtain the validity of all
quasi-Farkas-certificates in the $(\alpha_N^\top x\geq \gamma_N+1)$-branch
(depending on the choices for the right-hand-sides for all
disjunctions used at ancestors of $N$).
Assume that we have already expressed the simultaneous validity of the
quasi-Farkas-certificates in the $(\alpha_N^\top x\geq \gamma_N+1)$-branch
via a monotone real circuit computing values in $\{0,1\}$. The following
result states we can efficiently compute the smallest possible value of
$\gamma_N$ making these certificates valid via a monotone real circuit:

\begin{lemma}{(Oblivious Binary Search Lemma)}\label{lem:obl_bin}
  For a monotone real circuit~$C$ computing values in $\{0,1\}$ with~$k$ inputs,
  $\Lambda_{\max}\in \R$
  such that $C(x_1,\dots, x_{k-1},\Lambda_{\max}) = 1$ for
  all~$x_1,\dots, x_{k-1}\in\R$, and any~$q\in \N$, there exists a
  monotone real circuit~$\tilde C$ of size~$\card C \cdot q$
  which computes
     \[b \colon (x_1,\dots,x_{k-1}) \mapsto \max \big\{\lambda\in \{0,\dots,2^{q}-1\}~|~ \text{$C(x_1,\dots,x_{k-1},\Lambda_{\max} -\lambda) = 1$}\big\}.\]
\end{lemma}
Note that  $\tilde C$ uses~$q$ invocations of~$C$ to find the
maximal $\lambda$ among the $2^q$ candidates in
$\{0,\dots,2^{q}-1\}$ causing $C$ to accept
on input $(x_1,\dots,x_{k-1},\Lambda_{\max} -\lambda)$,
which is the best we can reasonably expect while
treating $C$ as a black box. In fact, the
concrete form of the set of candidate values is irrelevant -- we can
search over any set of this cardinality with a circuit of the same
size. This reflects the fact that the definition of monotone circuits
does not make use of the arithmetical structure of the real numbers.
Finally, we note that the assumption
$C(x_1,\dots, x_{k-1},\Lambda_{\max}) = 1$ for
all~$x_1,\dots, x_{k-1}\in\R$ is purely for convenience and could be
replaced by one more invocation of~$C$ and a slightly less elegant
definition of $b$.

 A key observation is that we are able to choose the right-hand-side
 $\gamma_N$ of the disjunction at a node $N$ by querying the circuit
 corresponding to the smaller child subtree. We then have to invoke the
 circuit corresponding to the larger subtree only once, to ensure that
 this choice of $\gamma_N$ works for both subtrees. This will ensure
 the efficiency of our construction.  Formally, we have:

 \begin{cor}\label{cor:circ_bal}
  For~$\kappa$,~$\Lambda_{\min}$,~$\Lambda_{\max} \in \Z$ and monotone real circuit~$C_1$ and $C_2$ computing values in $\{0,1\}$ with~$k$ inputs each,
  such that $C_1(x_1,\dots, x_{k-1},\Lambda_{\max}) = 1$ for
  all~$x_1,\dots, x_{k-1}\in\R$ and 
  $C_2(x_1,\dots, x_{k-1},\kappa-\Lambda_{\min}) = 1$ for
  all~$x_1,\dots, x_{k-1}\in\R$, there exists a monotone real
  circuit~$\tilde C$ with~$k-1$ inputs
  which decides
  whether there exist integral values~$x_k$,~$x'_k \in \Z$
  with $x_k+x'_k =\kappa$, such that
  $C_1(x_1,\dots,x_k) = C_2(x_1,\dots,x_{k-1},x'_k) = 1$ of size
  $\card{C_1} \cdot \lceil \log(L +1)\rceil + \card{C_2}$,
  where~$L \coloneqq \Lambda_{\max}-\Lambda_{\min}$.
\end{cor}
Note that above lemma and its corollary are the reason we consider
monotone real circuit complexity. We do not see a way in which binary
monotone circuits can achieve a similar result (when the inputs
are provided in binary encoding).

Finally, for the proof of
Theorem~\ref{thm:qfeas_real_monotone_interpolation}, we will construct
a monotone real circuit which on input~$z\in Z_1\cup Z_2$ decides
whether after fixing the variables~$z$ in~(\ref{eq_interpol_templ}),
there exists a branch-and-bound tree~$T^{P(z)}$ for
$P(z) = {\{x \in \{0,1\}^{n_1} ~|~Ax\leq a-Cz\}}$ as in
Lemma~\ref{lem:conformal_cert_interpolation} by efficiently searching
the space of possible right-hand-sides using
Corollary~\ref{cor:circ_bal}.  Clearly, if such right-hand-sides
exist, we have~$z\in Z_2$ and~$z\in Z_1$ otherwise.
\section{Proofs}
\label{sec:proofs}

\subsection{Proof of Lemmas~\ref{lem:conformal_interpolation} and~\ref{lem:conformal_cert_interpolation}}

To showcase our technique, we begin by showing Lemma~\ref{lem:conformal_interpolation}:

\begin{proof}[Proof of Lemma~\ref{lem:conformal_interpolation}]
  Let~$P \define \{x~|~Ax\leq a\}$
  and~$Q \define \{y~|~By\leq b\}$.  We will proceed by induction
  on~$\card T$. In case~$\card T =1$, we have $P\times Q = \emptyset$,
  i.e., $P =\emptyset$ or~$Q=\emptyset$, say~$P =\emptyset$.  Then
  there exists a branch-and-bound tree~$T'$ for~$P$
  with~$\card{T'} = 1$ which conforms to~$T$, since their common underlying
  directed graph does not contain any internal nodes.
  
  Assume~$\card T > 1$ and
  that~$\alpha^\top x+\beta^\top y\leq \delta \lor \alpha^\top x+\beta^\top y\geq
  \delta+1$ is the topmost disjunction of~$T$.  Let~$N_\leq$ denote
  the~$\leq$-child of the root of~$T$ and~$T(N_\leq)$ the subtree
  of~$T$ rooted at~$N_\leq$. For any~$\gamma \in \Z$ the
  subtree~$T(N_\leq)$ is a valid branch-and-bound tree
  for
  \[
    (P\times Q)_{\leq \gamma} \coloneqq (P\cap \{\alpha^\top x\leq
    \delta +\gamma\}) \times (Q\cap \{\beta^\top y\leq -\gamma\}),
  \]
  since the two added constraints imply~$\alpha^\top x+\beta^\top y\leq \delta$.

  The induction hypothesis then implies that either \ref{case:confP} or \ref{case:confQ} holds for~$T(N_\leq)$ and~$(P\times Q)_{\leq \gamma}$.
  This allows us to define~$(\mu_\leq(\gamma))_{\gamma \in \Z}$ by
  \begin{align*}
  \mu_\leq(\gamma) \coloneqq
  \begin{cases}
      -1, & \text{if case~\ref{case:confP} holds for~$T(N_\leq)$ and~$(P\times Q)_{\leq \gamma}$, but case~\ref{case:confQ} does not,}\\
      0, & \text{if cases~\ref{case:confP} and~\ref{case:confQ} both hold for~$T(N_\leq)$ and~$(P\times Q)_{\leq \gamma}$,} \\
      1, & \text{if case~\ref{case:confQ} holds for~$T(N_\leq)$ and~$(P\times Q)_{\leq \gamma}$, but case~\ref{case:confP} does not.}\\
  \end{cases}
  \end{align*}

  Indeed, by the induction hypothesis, we have defined~$\mu_\leq$ for all possible cases.
  It is easy to see that~$\mu_\leq$ is non-decreasing. Moreover,~$\mu_\leq$ is neither identically~$-1$~nor identically~$1$, since sufficiently extreme values of~$\gamma$ can render both~$(P \cap \{\alpha^\top x\leq \delta+\gamma\})$ and~$(Q\cap \{\beta^\top y\leq -\gamma\})$ empty (and hence any branch-and-bound tree is valid for them).
  
  Similarly, we let $N_\geq$ denote
  the~$\geq$-child of the root of~$T$ and~$T(N_\geq)$ the subtree
  of~$T$ rooted at~$N_\geq$. We define~$(P\times Q)_{\geq \gamma} \coloneqq (P\cap \{\alpha^\top x\geq \delta +\gamma +1\}) \times (Q\cap \{\beta^\top y\geq -\gamma \})$ and then~$(\mu_\geq(\gamma))_{\gamma \in \Z}$ by
  \begin{align*}
  \mu_\geq(\gamma) \coloneqq
  \begin{cases}
      -1, & \text{if case~\ref{case:confP} holds for~$T(N_\geq)$ and~$(P\times Q)_{\geq \gamma}$, but case~\ref{case:confQ} does not,}\\
      0, & \text{if cases~\ref{case:confP} and~\ref{case:confQ} both hold for~$T(N_\geq)$ and~$(P\times Q)_{\geq \gamma}$,} \\
      1, & \text{if case~\ref{case:confQ} holds for~$T(N_\geq)$ and~$(P\times Q)_{\geq \gamma}$, but case~\ref{case:confP} does not.}\\
  \end{cases}
  \end{align*}
  Note that~$\mu_\geq(\gamma)$ is non-increasing and neither
  identically~$-1$~nor identically~$1$. We remark that this definition
  is not symmetric in~$P$ and~$Q$ (the `$+1$' goes with~$P$).
  
  The noted properties of~$\mu_{\leq}$ and~$\mu_\geq$ imply that at
  least one of the following cases hold: Either (i) there
  is~$\gamma \in \Z$, such
  that~$\mu_\leq(\gamma)\leq 0$ and~$\mu_\geq(\gamma) \leq 0$, or (ii) there
  is~$\gamma\in\Z$, such
  that~$\mu_\leq(\gamma) = \mu_\geq(\gamma-1) =1$.

  If there exists~$\gamma$ as in (i), we construct~$T^P$ as
  desired by branching
  on~$\alpha^\top x\leq \delta + \gamma \lor \alpha^\top x\geq \delta +
  \gamma+1$ and attaching to the resulting children the trees
  for~$(P \cap \{\alpha^\top x\leq \delta+\gamma\})$
  and~$(P \cap \{\alpha^\top x\geq \delta+\gamma+1\})$ conforming
  to~$T(N_\leq)$ and~$T(N_\geq)$ for which existence is guaranteed
  by~$\mu_\leq(\gamma)\leq 0$ and $\mu_\geq(\gamma) \leq 0$.

  Otherwise, there exists~$\gamma$ as in (ii) and we
  construct~$T^Q$ as desired by branching
  on~$\beta^\top y \leq -\gamma \lor \beta^\top y\geq -\gamma+1$ and attaching
  to the resulting children the trees
  for~$ (Q\cap \{\beta^\top y\leq -\gamma\})$
  and~$(Q\cap \{\beta^\top y\geq -\gamma+1\})$ conforming to~$T(N_\leq)$
  and~$T(N_\geq)$ for which existence is guaranteed
  by~$\mu_\leq(\gamma) = \mu_\geq(\gamma-1) =1$.
\end{proof}

The proof of Lemma~\ref{lem:conformal_cert_interpolation} follows the
same arguments as Lemma~\ref{lem:conformal_interpolation}, but we now
need to track that the condition on the quasi-Farkas certificates
holds as well.

\begin{proof}[Proof of Lemma~\ref{lem:conformal_cert_interpolation}]
  Let~$P \define \{x~|~Ax\leq a\}$ and~$Q \define \{y~|~By\leq b\}$, i.e.,
  \[P\times Q = \left\{
    \begin{pmatrix}
     x\\ y
   \end{pmatrix}~\Big|~
  \begin{pmatrix}
     A & 0\\ 0 & B
   \end{pmatrix} 
    \begin{pmatrix}
     x\\ y
   \end{pmatrix}\leq   
    \begin{pmatrix}
     a\\ b
   \end{pmatrix}\right\}.\]
We will proceed by induction on~$\card T$.  If we have~$\card T =1$,
then the root~$r$ of~$T$ is the unique node of~$T$ and~$T$ does not
branch on any disjunctions.  Hence, there is a valid
Farkas-certificate~$f^r$ attached to~$r$. Let~$f^r_P$ denote the
projection of~$f^r$ onto constraints belonging to~$P$ and~$f^r_Q$ the
projection of~$f^r$ onto constraints belonging to~$Q$.  We then
have
\[
  \begin{pmatrix}
    {f^r_P} \\ {f^r_Q}
  \end{pmatrix}^\top
  \begin{pmatrix}
    A & 0\\ 0 & B
  \end{pmatrix}
  = 0, \quad
  \begin{pmatrix}
    {f^r_P} \\ {f^r_Q}
  \end{pmatrix}^\top
  \begin{pmatrix}
    a \\ b
  \end{pmatrix} = (f^r_P)^\top a + (f^r_Q)^\top b < 0.
\]
It follows that $(f^r_P)^\top A = 0$, $(f^r_Q)^\top B = 0$, and
that $(f^r_P)^\top a < 0$ or $(f^r_Q)^\top b< 0$. Hence,~${f^r_P}$
   is a Farkas-certificate for the infeasibility of~$P$ or~$f^r_Q$ is
   a Farkas-certificate for the infeasibility of~$Q$.  Thus, the
   branch-and-bound tree with a single leaf and certificate~${f^r_P}$
   for this leaf is a quasi-certified branch-and-bound tree for~$P$
   relative to~$U_P$ conforming to~$T$ or the branch-and-bound tree
   with a single leaf and certificate~${f^r_Q}$ is a quasi-certified
   branch-and-bound tree for~$Q$ relative for~$U_Q$ conforming to~$T$.
   
   Assume~$\card T >1$ and that
   $\alpha^\top x+\beta^\top y\leq \delta \lor \alpha^\top x+\beta^\top y\geq \delta+1$
   is the topmost disjunction of~$T$.  Let~$N_\leq$ denote the
   $\leq$-child of the root of~$T$ and~$T(N_\leq)$ the subtree of~$T$
   rooted at~$N_\leq$. For any~$\gamma \in \Z$ the subtree~$T(N_\leq)$
   is a valid branch-and-bound tree for
   \[
     (P\times Q)_{\leq \gamma} \define (P\cap \{\alpha^\top x\leq \delta
     +\gamma\}) \times (Q\cap \{\beta^\top y\leq -\gamma\}),
   \]
   since the two
   added constraints imply $\alpha^\top x+\beta^\top y\leq \delta$.
   In order to use the induction hypothesis, we must turn~$T(N_\leq)$
   into a quasi-certified branch-and-bound tree for
   $(P\times Q)_{\leq \gamma}$ relative to~$U$. For every leaf~$L$ in~$T$
   which is a descendant of~$N_\leq$, there is a
   quasi-Farkas-certificate~$f^L$ for the associated subproblem
   $T_L(P\times Q)$.  To obtain the problem
   $T(N_\leq)_L((P\times Q)_{\leq \gamma})$ associated to~$L$ as a
   leaf of the branch-and-bound tree~$T(N_\leq)$ for
   $(P\times Q)_{\leq \gamma}$, we have to replace the constraint
   $\alpha^\top x+\beta^\top y\leq \delta$, which we denote by~$\eta$, by
   $\alpha^\top x\leq \delta + \gamma$ and $\beta^\top y\leq -\gamma$, which we
   denote by~$\eta_P$ and~$\eta_Q$, respectively.  We define
   $\tilde f^L$ indexed by the constraints of $T(N_\leq)_L((P\times Q)_{\leq \gamma})$ by
   \[
    \tilde f^L_\nu = 
    \begin{cases}
      f^L_\nu, & \text{if $\nu \not\in \{\eta_P,\eta_Q\}$,} \\
      f^L_\eta, & \text{if $\nu \in \{\eta_P,\eta_Q\}$.}
    \end{cases}
  \]

  We claim that~$\tilde f^L$ is a Farkas-certificate for~$T(N_\leq)_L((P\times Q)_{\leq \gamma})$,
  if~$f^L$ is one for $T_L(P\times Q)$.
  Indeed, let us consider the system~$M
  \binom{x}{y} \leq m$, with
  \[
    M\coloneqq
  \begin{pmatrix}
    A & 0 \\
    0 & B \\
    E_1 &E_2\\
  \end{pmatrix}
  \text{ and }
  m \coloneqq
  \begin{pmatrix}
    a \\ b \\ e\\
  \end{pmatrix},
  \]
such that with~$E \define [E_1, E_2]$
we have that $E \binom{x}{y} \leq e$ are the constraints of $T_L(P\times Q)$ which come from
branching decisions in~$T$ and additionally~$\eta_P$ and
$\eta_Q$. Moreover, extend \smash{$f^L$} and \smash{$\tilde f^L$} with zeros, so
that the following computations are well-defined:
\begin{align*}
  (f^L-{\tilde f}^L)^\top M &= f^L_\eta M_{\eta} - {\tilde f}_{\eta_P}^L M_{\eta_P} -{\tilde f}_{\eta_Q}^L M_{\eta_Q} \\
  &= f^L_\eta M_{\eta} - f_{\eta}^L M_{\eta_P} -f_{\eta}^L M_{\eta_Q} \\&= f^L_\eta (M_\eta- M_{\eta_P} -M_{\eta_Q}) = f^L_\eta\cdot 0 = 0     
\end{align*}
and
\begin{align*}
  (f^L-{\tilde f}^L)^\top m &= f^L_\eta m_{\eta} - {\tilde f}_{\eta_P}^L m_{\eta_P} -{\tilde f}_{\eta_Q}^L m_{\eta_Q} \\
  &= f^L_\eta m_{\eta} - f_{\eta}^L m_{\eta_P} -f_{\eta}^L m_{\eta_Q} = 0.
\end{align*}
Hence, we have
\begin{align*}
  ({\tilde f}^L)^\top M = ({f}^L)^\top M = 0 \text{ and } ({\tilde f}^L)^\top m = ({f}^L)^\top m < 0.
\end{align*}
Thus,~$\tilde f^L$ is a valid Farkas-certificate for the infeasibility
of $T(N_\leq)_L((P\times Q)_{\leq \gamma})$, if~$f^L$ is one for
$T_L(P\times Q)$.  Furthermore, note that if there is an edge~$e$ in
the path from the root to~$L$ in~$T$ labeled with an inequality
$\tilde{\alpha}^\top x + \tilde{\beta}^\top y \leq \tilde \delta $, such that
$U \cap \{\tilde{\alpha}^\top x + \tilde{\beta}^\top y \leq \tilde \delta \} =
\emptyset$ and~$e$ is not the edge between the root and~$N_\leq$, then
$e$ is also contained in the path between the root and~$L$ in
$T(N_\leq)$.

Hence, if
$U \cap \{\alpha^\top x + \beta^\top y \leq \delta \} \neq \emptyset$, then
$T(N_\leq)$ with the quasi-Farkas-certificates constructed above is
indeed a valid quasi-certified branch-and-bound tree for
$(P\times Q)_{\leq\gamma}$ relative to~$U$ and thus, the induction
hypothesis implies that~\ref{case:cert_confP} or~\ref{case:cert_confQ} holds for
$T(N_\leq)$ as a branch-and-bound tree for $(P\times Q)_{\leq \gamma}$
relative to~$U$.  Moreover, if
$U \cap \{\alpha^\top x + \beta^\top y \leq \delta\} = \emptyset$, then we have
$U_P \cap \{\alpha^\top x \leq \delta +\gamma\} = \emptyset$ or
$U_Q \cap \{\beta^\top y \leq -\gamma\} = \emptyset$.

Define the following cases:

  \begin{enumerate}[label=(\alph*$^\prime$)]
  \item \label{case:confP'} Case~\ref{case:cert_confP} holds for~$T(N_\leq)$ and~$(P\times Q)_{\leq \gamma}$ or $U_P \cap \{\alpha^\top x \leq \delta +\gamma\} = \emptyset$.
   \item \label{case:confQ'} Case~\ref{case:cert_confQ} holds for~$T(N_\leq)$ and~$(P\times Q)_{\leq \gamma}$ or $U_Q \cap \{\beta^\top y \leq -\gamma\} = \emptyset$.
   \end{enumerate}

  This allows us to define $(\mu_\leq(\gamma))_{\gamma \in \Z}$ by
  \begin{align*}
  \mu_\leq(\gamma) \coloneqq
  \begin{cases}
      -1, & \text{if case~\ref{case:confP'} holds,  but case~\ref{case:confQ'} does not,}\\
      0, & \text{if cases~\ref{case:confP'} and~\ref{case:confQ'} both hold,} \\
      1, & \text{if case~\ref{case:confQ'} holds, but case~\ref{case:confP'} does not.}\\
  \end{cases}
  \end{align*}

  Indeed, by the induction hypothesis and our previous considerations,
  we have defined~$\mu_\leq$ for all possible cases.  It is easy to
  see that~$\mu_\leq$ is non-decreasing. Moreover,~$\mu_\leq$ is
  neither identically~$-1$ nor identically~$1$, since sufficiently
  extreme values of~$\gamma$ can render both
  $(U_P \cap \{\alpha^\top x\leq \delta+\gamma\})$ and
  $(U_Q\cap \{\beta^\top y\leq -\gamma\})$ empty.
  
  Similarly, one defines
  $(P\times Q)_{\geq \gamma} \coloneqq (P\cap \{\alpha^\top x\geq \delta
  +\gamma +1\}) \times (Q\cap \{\beta^\top y\geq -\gamma \})$ and then the
  following cases:
  \begin{enumerate}[label=(\alph*$^{\prime\prime}$)]
  \item \label{case:confP''} Case~\ref{case:cert_confP} holds for~$T(N_\geq)$ and~$(P\times Q)_{\geq \gamma}$ or $U_P \cap \{\alpha^\top x \geq \delta +\gamma +1\} = \emptyset$.
   \item \label{case:confQ''} Case~\ref{case:cert_confQ} holds for~$T(N_\geq)$ and~$(P\times Q)_{\geq \gamma}$ or $U_Q \cap \{\beta^\top y \geq -\gamma \} = \emptyset$,
   \end{enumerate}
and finally  $(\mu_\geq(\gamma))_{\gamma \in \Z}$ by
 \begin{align*}
  \mu_\geq(\gamma) \coloneqq
  \begin{cases}
    -1, & \text{if case~\ref{case:confP''} holds, but case~\ref{case:confQ''} does not,}\\
    0, & \text{if cases~\ref{case:confP''} and~\ref{case:confQ''} both hold,} \\
    1, & \text{if case~\ref{case:confQ''} holds, but case~\ref{case:confP''} does not.}\\
    \end{cases}
 \end{align*}
 One checks easily
 that~$\mu_\geq(\gamma)$ is non-increasing and neither identically~$-1$
 nor identically~$1$.
  
  The noted properties of~$\mu_{\leq}$ and~$\mu_\geq$ imply that at
  least one of the following cases hold: (i) There is~$\gamma \in \Z$,
  such that~$\mu_\leq(\gamma)\leq 0$ and~$\mu_\geq(\gamma) \leq 0$, or (ii)
  there is~$\gamma\in\Z$, such that
  $\mu_\leq(\gamma) = \mu_\geq(\gamma-1) =1$.

  If there exists~$\gamma$ as in case (i), i.e., cases~\ref{case:confP'} and~\ref{case:confP''} hold for this~$\gamma$, we construct~$T^P$ as
  desired by branching on
  $\alpha^\top x\leq \delta + \gamma \lor \alpha^\top x\geq \delta +
  \gamma+1$ and attaching to the resulting children~$N_\leq$
  and~$N_\geq$ the following trees: If~\ref{case:cert_confP} holds for~$T(N_\leq)$
  and $(P\times Q)_{\leq \gamma}$, we attach the
  quasi-certified branch-and-bound tree~$T(N_\leq)^P$ for
  $P\cap\{\alpha^\top x \leq \delta +\gamma\}$ relative to~$U_P$
  conforming to~$T(N_\leq)$ given by~\ref{case:cert_confP}.  While doing so,
  we can keep the quasi-Farkas-certificates attached to leaves
  in~$T(N_\leq)^P$. For this, note
  that the set of constraints describing the problem
  $(T(N_\leq)^P)_L(P\cap\{\alpha^\top x \leq \gamma+\delta\})$
  associated to a
  leaf~$L$ in the branch-and-bound tree~$T(N_\leq)^P$ for
  $P\cap\{\alpha^\top x \leq \gamma+\delta\}$ is identical to the
  constraints describing the problem~$(T^P)_L(P)$ associated to~$L$ as
  a leaf of the branch-and-bound tree~$T^P$ for~$P$. If~\ref{case:cert_confP}
  does not hold for~$T(N_\leq)$ and $(P\times Q)_{\leq \gamma}$, we have
  $U_P \cap \{\alpha^\top x \leq \delta +\gamma\} = \emptyset$, and we
  attach an arbitrary, not necessarily valid quasi-certified
  branch-and-bound tree conforming to~$T(N_\leq)$.

  Similarly, if~\ref{case:cert_confP} holds for~$T(N_\geq)$ and
  $(P\times Q)_{\geq \gamma}$, then we attach to~$N_\geq$ the
  quasi-certified branch-and-bound tree~$T(N_\geq)^P$ relative
  to~$U_P$ conforming to~$T(N_\geq)$ given by~\ref{case:cert_confP}. Otherwise, we have
  $U_P \cap \{\alpha^\top x \geq \delta +\gamma\ +1\} = \emptyset$,
  and we attach an arbitrary, not necessarily valid quasi-certified branch-and-bound
  tree for~$P$ conforming to~$T(N_\geq)$. Choose
  quasi-Farkas-certificates as in the previous case.

  It is then easy to see that~$T^P$ is conforming to~$T$.  It remains
  to check that~$T^P$ is a valid quasi-certified branch-and-bound tree
  for~$P$ relative to~$U_P$. For this consider a leaf~$L$ of~$T^P$ in
  the subtree rooted at~$N_\leq$. If we have
  $U_P \cap \{\alpha^\top x \leq \delta +\gamma\} = \emptyset$, there
  is nothing to check for~$f^L$.  Similarly, if there is an edge in~$T(N_\leq)^P$
  on the path from the root to~$L$ labeled with an
  inequality $\tilde \alpha^\top x\leq \tilde \delta$, such that
  $U_P \cap \{\tilde \alpha^\top x \leq\tilde \delta\} = \emptyset$,
  then the same inequality appears on the path from the root to~$L$ in~$T^P$.
  The only remaining case is that the quasi-Farkas-certificate~$f^L$
  attached to~$L$ in~$T(N_\leq)^P$ is a Farkas-certificate for the associated problem
  $(T(N_\leq)^P)_L(P\cap \{\alpha^\top x\leq \delta+\gamma\})$. But then~$f^L$
  is also a Farkas-certificate for the problem~$(T^P)_L(P)$
  associated to~$L$ in~$T^P$, which is the same problem.  For leaves
  in the subtree rooted at~$N_\geq$ we proceed similarly.

  In case (ii) we proceed analogously.
\end{proof}

\subsection{Proofs of Lemma~\ref{lem:obl_bin} and Corollary~\ref{cor:circ_bal}}

\begin{proof}[Proof of Lemma~\ref{lem:obl_bin}]
  For ease of notation, we write $p \coloneqq q-1$ instead.  Thus we have to compute
     \[b \colon (x_1,\dots,x_{k-1}) \mapsto \max \big\{\lambda\in \{0,\dots,2^{p+1}-1\}~|~ \text{$C(x_1,\dots,x_{k-1},\Lambda_{\max} -\lambda) = 1$}\big\}.\]

     We will construct a monotone real circuit~$\tilde C$ of the
     desired size which works in~$p+1$ phases~$0,\dots,p$.  For each
     phase~$i$, there will be a gate~$h_i$ in~$\tilde C$ representing
     the state of computation after phase~$i$.  The gate~$h_i$ will
     compute the function
     \[
       b_i(x) \define \lfloor b(x) \rfloor_{2^{p-i}},
     \]
     where $x = (x_1,\dots,x_{k-1})$ and~$\lfloor \cdot \rfloor_{2^{p-i}}$ denotes rounding down to the nearest integer divisible by~$2^{p-i}$.
     Clearly, this
     suffices, since~$h_p$ then computes the desired
     function. Moreover,~$b_0$ can be computed by a copy~$C_0$ of~$C$,
     which receives~$(x,\Lambda_{\max}-2^p )$ where the output gate~$h_0$
     is modified to compute~$b_0(x) = 2^{p}C(x,\Lambda_{\max}-2^p )$.

  To construct the part of~$\tilde C$ belonging to phase~$i > 0$, we will rely on the recurrence
  \[
    b_{i}(x) = b_{i-1}(x) + 2^{p-i}C(x,\Lambda_{\max}-b_{i-1}(x)-2^{p-i}).
  \]
  Unfortunately, this recurrence is not necessarily monotone in~$x$ and~$b_{i-1}$ due to the sign on the second occurrence of~$b_{i-1}(x)$. 
  However, since it is immediate from the definition that~$b_{i-1}$ is divisible by~$2^{p-(i-1)}$, we might as well consider the recurrence
  \begin{equation}\label{eq:rec_bi}
    b_{i} = \lfloor b_{i-1}\rfloor_{2^{p-(i-1)}} + 2^{p-i}C(x,\Lambda_{\max}-\lfloor b_{i-1} \rfloor_{2^{p-(i-1)}} -2^{p-i}).
  \end{equation}
  Note that we drop the dependence of~$b_i$ and~$b_{i-1}$ on~$x$ to improve readability.
  
Unfortunately for us, the latter recurrence -- while clearly monotone in~$x$
and~$b_{i-1}$
 -- still does not provide an obvious monotone real circuit,
since the last summand applies a non-monotone function to~$b_{i-1}$.

Thus, we have to give a version of~$C$, which passes along the old
bound~$b_{i-1}$ from the~$k$-th input gate to the output gate in the
higher order bits, together with how~$C$ behaves on input
$(x,\Lambda_{\max}-b_{i-1}-2^{p-i})$ in the lower order bits, in order
to make our computation monotone.

We begin by assuming that every non-input gate in~$C$ applies a
function~$f$ with $\range f\subseteq (0,\tfrac{1}{2})$. This can be achieved by
post-composing the function applied at every gate with the monotone bijection
\[\varphi \colon \R \rightarrow (0,\tfrac{1}{2}), \quad y\mapsto (\arctan(y)
  + \tfrac{\pi}{2})/2\pi\]
and pre-composing every function applied at a gate,
which takes as input such a modified gate, with~$\varphi^{-1}$ for the
respective input.
Let~$C'$ be the monotone real circuit obtained from~$C$ this way. We
then have $C'(x) = \varphi(C(x))$.

Next, we introduce a gate~$\tilde g_k$, which provides both the~$k$-th
input (which is supposed to be~$b_{i-1}$) and the transformed value of
$\Lambda_{\max}- b_{i-1} -2^{p-i}$, i.e,
both inputs to~$\tilde g_k$ are the~$k$-th input gate~$g_k$ of~$C'$
and the function applied at~$\tilde g_k$ is
\[
  f^i(g_k,g_k) \coloneqq \lfloor g_k \rfloor_{2^{p-(i-1)}} +
  \varphi\big(\Lambda_{\max}-\lfloor g_k \rfloor_{2^{p-(i-1)}}
  -2^{p-i}\big).
\]
Then, for every gate~$g$ which uses~$g_k$ as input in
$C'$, we let~$g$ use~$\tilde g_k$ as input instead. We note~$f^i$ is
non-decreasing: If an increase of~$g_k$ would cause the summand
involving~$\varphi$ to decrease, then it decreases by at most
$\frac 1 2$, but the other summand then increases by at least~$1$.

Let~$S$ denote the set of gates in~$C'$ which are a descendant of
$g_k$ (hence now of~$\tilde g_k$) and consider a gate~$g$ in~$S$
with predecessors~$g_1$ and~$g_2$, such that~$g_1\in S$, but
$g_2\not\in S$. We then replace the function~$f_g(g_1,g_2)$ applied at
$g$ by the function
\begin{align*}
  f^i_g(g_1,g_2) \coloneqq& \lfloor g_1 \rfloor_{2^{p-(i-1)}} +
  f_g\big(\{g_1\}_{2^{p-(i-1)}},\,g_2\big),
\end{align*}
where $\{g_1\}_{2^{p-(i-1)}} \coloneqq g_1 -\lfloor g_1 \rfloor_{2^{p-(i-1)}}$. We check that~$f^i_g$ is non-decreasing: If an increase of~$g_1$ would cause~$\{g_1\}$ to decrease, then $f_g(\{g_1\}_{2^{p-(i-1)}},\,g_2)$ decreases by at most~$\tfrac{1}{2}$, since $\range f_g \subseteq (0,\tfrac{1}{2})$, while~$\lfloor g_1 \rfloor_{2^{p-(i-1)}}$ increases by at least~$1$.

If~$g$ is a gate with both predecessor~$g_1$ and~$g_2$ in~$S$, we replace the function~$f_g$ applied at~$g$ by
\begin{align*}
   f^i_g(g_1,g_2) \coloneqq
  (\lfloor g_1 \rfloor_{2^{p-(i-1)}} +\lfloor g_2 \rfloor_{2^{p-(i-1)}})/2 
  +
    f_g\big(\{g_1\}_{2^{p-(i-1)}},\,\{g_2\}_{2^{p-(i-1)}}\big),
\end{align*}
Once again,~$f_g^i$ is non-decreasing by an analogous argument.

Let~$C_i$ denote the monotone real circuit obtained by applying these
modifications to~$C'$.  For a gate~$g'$ in~$C'$, let~$g^i$ denote the
corresponding gate of~$C_i$.  It is then easy to show by induction
along a topological order on the gates of~$C'$ (or equivalently~$C_i$)
that
\[
  g^i(x,b_{i-1}) =
  \begin{cases}
    g'\big(x,\,\Lambda_{\max}- b_{i-1} -2^{p-i}\big), &\text{if $g\not\in S$},\\
    b_{i-1}+ g'\big(x,\,\Lambda_{\max}- b_{i-1} -2^{p-i}\big)  &\text{if $g\in S$.}\\
  \end{cases}
\]
We note that this holds for the input gates $g_1,\dots g_{k-1}$ and
$\tilde g_k$. For a non-input gate~$g'$ in $C'$,
consider for example the case where $g$ has predecessors~$g'_1$ and
$g'_2$, such that~$g'_1\in S$ and~$g'_2\not\in S$. Then a straight-forward
computation yields
\begin{align*}
  g^i(x,b_{i-1}) &=  f^i_g\big(g_1^i(x,b_{i-1}), g_2^i(x,b_{i-1})\big) \\
                 &= \lfloor g_1^i(x,b_{i-1})\rfloor_{2^{p-(1-i)}} + f_g\big(\{g_1^i(x,b_{i-1})\}_{2^{p-(1-i)}},\,g^i_2(x,b_{i-1})\big) \\
                 &= \lfloor b_{i-1}+ g_1'\big(x,\,\Lambda_{\max}- b_{i-1} -2^{p-i}\big) \rfloor_{2^{p-(1-i)}} \\
                 & \quad + f_g\big(\{ b_{i-1}+ g_1'(x,\,\Lambda_{\max}-b_{i-1}-2^{p-i})\}_{2^{p-(i-1)}},\, g_2'(x,\Lambda_{\max}-b_{i-1} -2^{p-i})\big) \\
                 &= b_{i-1}+ f_g\big( g_1'(x,\,\Lambda_{\max}- b_{i-1}  -2^{p-i}),\, g_2'(x,\Lambda_{\max}-b_{i-1} -2^{p-i})\big) \\
                  &= b_{i-1}  + g'\big(x,\,\Lambda_{\max}- b_{i-1}  -2^{p-i}\big)
\end{align*}
as desired, where the first and last identity are due to the definition of the computed function,
the second due to the definition of~$f_g^i$, the third due to the induction hypothesis and the fourth
due to $b_{i-1} = \lfloor b_{i-1} \rfloor_{2^{p-(i-1)}}$ and $\range g \in(0,\frac 1 2)$.  The other cases are analogous.

By considering this identity for the output gate~$h^i$ of~$C_i$ (note
that we can assume to be in the second case), we obtain
\begin{align*}
  C_i(x,b_{i-1}) &=  b_{i-(i-1)}+ C'(x,\Lambda_{\max}- b_{i-1} -2^{p-i}) \\
                 &= b_{i-(i-1)}+ \varphi (C(x,\Lambda_{\max}- b_{i-1} -2^{p-i})). 
\end{align*}
Thus, by our recurrence~(\ref{eq:rec_bi}), if we
post-compose the function applied at~$h_i$ with
\[y \mapsto
  \begin{cases}
    \lfloor y\rfloor_{2^{p-(i-1)}} &\text{if $\{ y\}_{2^{p-(i-1)}} < \varphi(1)$}, \\
    \lfloor y\rfloor_{2^{p-(i-1)}}+2^{p-i} &\text{if $\{ y\}_{2^{p-(i-1)}} \geq \varphi(1)$}, \\
  \end{cases}
\]
we obtain a monotone real circuit which on input~$(x, b_{i-1})$
computes~$b_i$. For this, note that $\{C_i(x,b_{i-1})\}_{2^{p-(i-1)}}  = \varphi(C(x,\Lambda_{\max}- b_{i-1} -2^{p-i})) < \varphi(1)$ is
equivalent to $C(x,\Lambda_{\max}-b_{i-1}-2^{p-i}) = 0$.
Then~$\tilde C$ can be constructed in the obvious
way, i.e., by sequentially using the constructed circuits
$C_0,\dots,C_p$ to compute the values $b_0(x),\dots,b_p(x) = b(x)$.

For the size bound, observe that we have used~$q = p+1$ copies of~$C$
and that the introduced auxiliary gates~$\tilde g_k$ can be eliminated
from the circuit, since the functions applied at the children
of~$\tilde g_k$ can instead be pre-composed with the function applied
at~$\tilde g_k$. 
\end{proof}

\begin{proof}[Proof of Corollary~\ref{cor:circ_bal}]
 Choose $q \coloneqq \lceil\log(L+1)\rceil $ and 
apply Lemma~\ref{lem:obl_bin} to~$C_1$ to construct~$\tilde C_1$, such that the output gate~$h$ of~$\tilde C_1$ computes

\[b \colon (x_1,\dots,x_{k-1}) \mapsto \max \big\{\lambda\in \{0,\dots,2^{q}-1\}~\big|~ \text{$C(x_1,\dots,\Lambda_{\max} -\lambda) = 1$}\big\}.\]
and use a copy of~$C_2$ to compute $C_2(x_1,\dots, x_{k-1}, \kappa - (\Lambda_{\max} - b(x_1,\dots,x_{k-1}))) $.

Clearly, if the output gate of this copy of~$C_2$ computes~$1$,
then $x_k= \Lambda_{\max}- b(x_1,\dots,x_{k-1})$ and $x'_k= \kappa - x_k$ satisfy $C_1(x_1,\dots,x_{k-1}, x_k) = C_2(x_1,\dots,x_{k-1}, x'_k) = 1$ and $x_k+x'_k = \kappa$. Otherwise, there are no possible such choices for~$x_k$ and~$x_k'$, since for $x_k < \Lambda_{\max}- b(x_1,\dots,x_{k-1})$ we have $C_1(x_1,\dots,x_{k-1}, x_k)=0$ and for $x_k \geq \Lambda_{\max}- b(x_1,\dots,x_{k-1})$ and $x'_k= \kappa - x_k$ we have $C_2(x_1,\dots,x_{k-1}, x'_k)=0$.
Hence, the constructed circuit decides the question posed in the corollary. 
\end{proof}

\subsection{Proof of Theorem~\ref{thm:qfeas_real_monotone_interpolation}}

  For ease of notation, we assume that the variable bounds in~(\ref{eq_interpol_templ})
  are incorporated into the constraints. Then the LP-relaxation of~(\ref{eq_interpol_templ})
  is given by:
\begin{equation*}
    \begin{pmatrix}A\\0\end{pmatrix}x +
    \begin{pmatrix}0\\B\end{pmatrix}y +
    \begin{pmatrix}C\\D\end{pmatrix}z \leq \begin{pmatrix}a\\b\end{pmatrix},\qquad
\end{equation*}
The basic structure of the proof is as follows: Given
$z\in Z_1 \cup Z_2$, where
$Z_1 = \{z \in \{0,1\}^{n_3}~|~ \exists x\in \{0,1\}^{n_1}\colon Ax
\leq a- Cz\}$ and
$Z_2= \{z \in \{0,1\}^{n_3}~|~ \exists y \in \{0,1\}^{n_2}\colon By
\leq b- Dz\}$, and a branch-and-bound tree~$T$ for the
infeasibility of~(\ref{eq_interpol_templ}), we compute
Farkas-certificates for the leaves of~$T$ and then obtain a
certified branch-and-bound tree~$\tilde T$ for
$P(z) \times Q(z) \coloneqq \{Ax\leq a- Cz\}\times \{Bx\leq b-
Dz\}$ by plugging in the values for~$z$ in the disjunctions used in
$T$.  Then at least one of the alternatives in
Lemma~\ref{lem:conformal_cert_interpolation} holds. However, since
$z\in Z_1 \cup Z_2$, exactly one of~$P(z)$ and~$Q(z)$ is
integer-feasible, and thus at most one of the alternatives in
Lemma~\ref{lem:conformal_cert_interpolation} holds. Clearly, we have
$z\in Z_2$ if and only if there exists a quasi-certified
branch-and-bound tree~$(\tilde T)^P$ for~$P(z)$ conforming to
$\tilde T$.  Thus, if we construct a monotone real circuit~$C$ that
given values for the~$z$ variables decides whether there exists such a
tree~$(\tilde T)^P$, then~$C$ separates~$Z_1$ and~$Z_2$. We work out
the details below:

\begin{proof}[Proof of Theorem~\ref{thm:qfeas_real_monotone_interpolation}]
  We again assume that variable bounds in~(\ref{eq_interpol_templ})
  are incorporated into the constraints as above.
  We begin by applying Theorem~\ref{thm:coeff_size} to our
  branch-and-bound tree~$T$ for~(\ref{eq_interpol_templ}) to obtain a
  certified branch-and-bound tree~$T'$ for~(\ref{eq_interpol_templ})
  with bounded coefficients: Note that the linear
  programming relaxation of~(\ref{eq_interpol_templ}) is contained in
  the ball $B_1^n(n) = \{x\in \R^n~|~ \lVert x \rVert_1 \leq n$\}, where
  $n\coloneqq n_1+n_2+n_3$. Hence, we can assume that for every
  disjunction $d^\top w \leq \delta \lor d^\top w \geq \delta+1$ used
  in~$T'$, we have
  $\max\{\lVert d\rVert_\infty,\card \delta\} \leq (10n^2)^{(n+2)^2}$
  and moreover we have
  $\card{T'}\leq (4n+5)\card T$. Then we fix some Farkas-certificates
  for~$T'$ which thus becomes a certified branch-and-bound tree.

  By fixing the values of~$z$ in the disjunctions used in~$T'$, we
  obtain a certified branch-and-bound tree~$\tilde T$ for
  $P(z) \times Q(z) \coloneqq\{Ax \leq a- Cz\} \times \{By \leq
  b- Dz\}$. Since $P(z)\times Q(z) \subseteq [0,1]^{n_1} \times
  [0,1]^{n_2}$, we may also consider~$\tilde T$ as a
  quasi-certified branch-and-bound tree for $P(z) \times Q(z)$
  relative to $[0,1]^{n_1} \times [0,1]^{n_2}$.
  
  Let~$\NN(T')$ denote the set of internal nodes of~$T'$ and let
  $\gamma \in \Z^{\NN(T')}$.
  Consider the not necessarily valid quasi-certified
  branch-and-bound tree $\tilde T^P(\gamma)$ for~$P(z)$ relative to~$[0,1]^{n_1}$,
  which has the same underlying
  directed tree as~$\tilde T$, and at a node~$N$ branches on the
  disjunction
  $\alpha^\top_N x \leq \gamma_N \lor \alpha^\top_N x \geq \gamma_N +1$,
  when~$\tilde T$ branches at~$N$ on the disjunction
  $\alpha^\top_N x +\beta_N^\top y \leq \delta \lor \alpha^\top_N x +\beta_N^\top y \geq
  \delta_N+1$. Similarly, the
  Farkas-certificate at a leaf~$L$ of~$\tilde T^P(\gamma)$ is the
  Farkas-certificate at leaf~$L$ of~$\tilde T$ with the entries
  corresponding to constraints from~$Q(z)$ removed.  We are interested
  in whether there exists a choice for~$\gamma$ for which~$\tilde T^P(\gamma)$
  is a valid quasi-certified branch-and-bound tree for~$P(z)$ relative to~$[0,1]^{n_1}$.
 
  For any candidate disjunction
  $\alpha^\top_N x \leq \gamma_N \lor \alpha^\top_N x \geq \gamma_N+1$ to be
  used at a node~$N$ in~$\tilde T^P(\gamma)$, the slab
  $\{x\in \R^{n_1}\colon \gamma_N \leq \alpha^\top_N x \leq \gamma_N
  +1\}$ has width
  $\frac {1} {\lVert \alpha_N\rVert_2} \geq \frac {1}{\sqrt{n_1}\lVert
    \alpha_N\rVert_\infty}$.
  Since
  \[
    \max\Big\{\Big(\frac{\alpha_N}{\lVert \alpha_N\rVert_2}\Big)^\top x ~|~ x\in [0,1]^{n_1} \Big\}
    -\min\Big\{\Big(\frac{\alpha_N}{\lVert \alpha_N\rVert_2}\Big)^\top x ~|~ x\in [0,1]^{n_1}\Big\}\leq \sqrt n_1,
  \]
  we have that at most
  $\sqrt{n_1} \cdot \sqrt{n_1}\lVert \alpha_N\rVert_\infty +2 \leq
  n(10n^2)^{(n+2)^2}+2$ of our slabs intersect~$[0,1]^{n_1}$.  Let~$L_{\min}^N$ denote the maximal value for~$\gamma$ for which
  $[0,1]^{n_1} \cap \{\alpha_N^\top x \leq \gamma\} = \emptyset$
  (cf.\ Figure~\ref{fig:Lminmax}).  Similarly, let~$L_{\max}^N$
  denote the minimal~$\gamma$ for which
  $[0,1]^{n_1} \cap \{\alpha_N^\top x \geq \gamma+1\} = \emptyset$.
  Moreover, let $L^N\coloneqq L_{\max}^N -L_{\min}^N$ and
  $L\coloneqq \max\{L^N~|~ \text{$N$ internal node of $T'$}\} \leq n(10n^2)^{(n+2)^2}+2$.

  \begin{figure}
    \centering
    \begin{subfigure}{0.4\linewidth}
    \centering
      \begin{tikzpicture}[scale=0.6]
    \foreach \i in {1,...,7} {
        \draw [very thin,gray] (\i,0) -- (\i -2,6); 
      }
      \draw (1,1) -- (5,1) -- (5,5) --(1,5) -- (1,1);
      \node[below, gray] at (1,0){$L_{\min}$};
       \node[below, gray ] at (6,0){$L_{\max}$};
       \node[below, gray] at (7.4,-0.6){$L_{\max} +1$};
       \node[] at (6.3,5){$[0,1]^{n_1}$};
     \end{tikzpicture}
     \caption{The hyperplanes $\alpha^\top x = \gamma$ for different values of $\gamma$.}
     \label{fig:Lminmax}
   \end{subfigure}
   \hfill
   \begin{subfigure}{0.51\linewidth}
    \centering
  \scalebox{1}{
    \begin{forest}
for tree={circle, draw, inner sep=0pt, minimum size=14pt, l sep=10pt, s sep = 35pt}
[$M_1$
  [$M_2$
    [,cross]
      [$M_3$
        [$N$
          [,cross]
          [$M_4$
            [,cross]
            [,cross]
          ]
        ]
        [,cross]
    ]
  ]
  [,cross]
]
\end{forest}
}
\caption{An example for~$T'$ for which $\V^+(N) = \{\gamma^+_{M_2}\}$, $\V^-(N) = \{\gamma^-_{M_1},\gamma^-_{M_3}\}$ and $\U(N) = \{\gamma^-_N,\gamma^+_N,\gamma^-_{M_4},\gamma^+_{M_4}\}$.}
\label{fig:variablesExample}
\end{subfigure}
\caption{Illustrations of $L_{\min}$, $L_{\max}$, $\U(N)$ and
  $\V(N)$. }
\label{fig:interpol_thm_defs}
\end{figure}
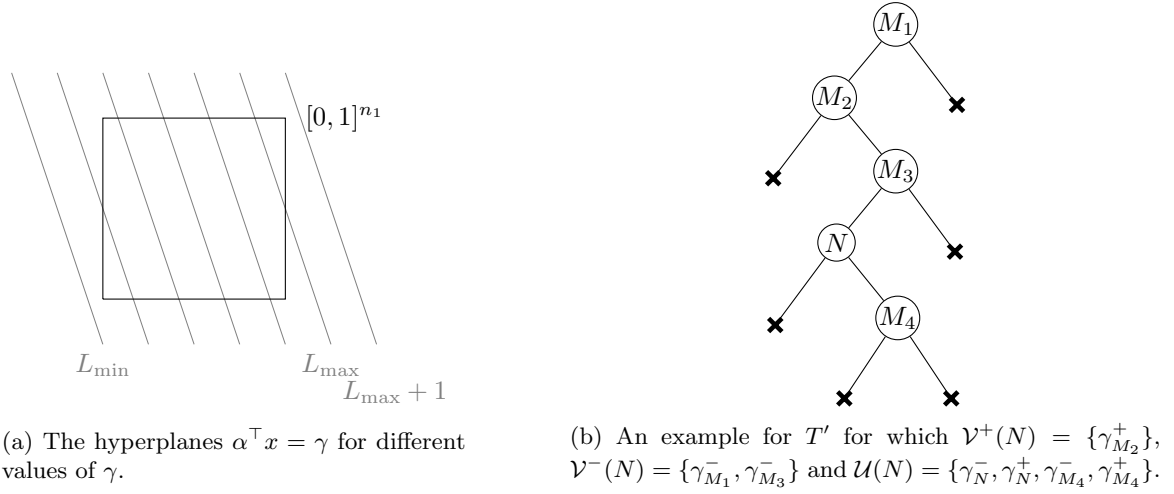

Then, for every internal node~$N$ of~$T'$, we introduce two variables
$\gamma_N^+$ and~$\gamma_N^-$. Variable~$\gamma_N^+$ represents how
far the right-hand-side of the disjunction at~$N$ is chosen away from
the lower bound~$L_{\min}^N$, while~$\gamma_N^-$ represents how
far the right-hand-side of the disjunction at~$N$ is chosen away from
the upper bound~$L_{\max}^N$. Thus, $\gamma_N = L_{\min}^N + \gamma_N^+ = L_{\max}^N - \gamma_n^-$. Hence, in order for the pair
$(\gamma_N^+,\gamma_N^-)$ to represent a valid right-hand-side for the
node~$N$, we must have~$\gamma_N^++\gamma_N^- = L^N$. Note that this
representation of~$\gamma_N$ allows us to work with the usual definition of
monotone real circuits and not deal with the case where a function is
non-increasing in an input variable.

  For every node~$N$ of~$T'$, let~$\anc(N)$ denote the set of proper ancestors of~$N$ (i.e.,
  excluding~$N$). Then define 
 \begin{align*}
   \V(N) \coloneqq \V^+(N) \cup \V^-(N) \coloneqq
    &\hphantom{\cup~} \left\{\gamma_M^+~\Big|~ \text{
     \begin{parbox}{18em}
       {$M\in \anc(N)$ and~$N$ is in the subtree rooted at the $\alpha_M^\top x \geq \gamma_M+1$-child of~$M$} 
     \end{parbox}} \right\} \\
   &  \cup  \left\{\gamma_M^-~\Big|~ \text{
     \begin{parbox}{18em}
       {$M\in \anc(N)$ and~$N$ is in the subtree rooted at the $\alpha_M^\top x \leq \gamma_M$-child of~$M$} 
     \end{parbox}}\right\}
 \end{align*}
 and
 \[
   \U(N) \coloneqq \{\gamma_M^-,\gamma_M^+ ~|~ M \text{ internal node of
     $T'$ and descendant of }N\}.
 \]
 For this definition, we consider~$N$ as a descendant of~$N$. See
 Figure~\ref{fig:variablesExample} for an example.
 
Then, for the sake of induction, we strengthen the statement of the theorem to:
 
\smallskip
 
 \begin{claim} For every node~$N$ of~$T'$, there exists a
 monotone real circuit~$C_N$ of size
 $\card{T'(N)} \cdot 2(\card{T'}+n_3)\cdot \log(L +1)^{\log
   \card{T'(N)}}$, which receives as inputs values for the variables~$\ZZ \cup \V(N)$, where $\ZZ \coloneqq \{z_1, \dots, z_{n_3}\}$
 and decides whether there exist values for the variables~$\U(N)$,
 which obey $\gamma_N^++\gamma_N^- =L^N$ and choosing
 \begin{align*}
   &\gamma_N \coloneqq L^N_{\min} +\gamma_N^+ &&\forall\, \gamma^+_N \in \U(N)\cup \V^+(N) \quad \text{and} \\
   &\gamma_N \coloneqq L^N_{\max} -\gamma_N^- &&\forall\, \gamma^-_N \in \V^-(N)
 \end{align*}
 turns every~$f^L$ attached to a leaf~$L$ in the corresponding
 subtree~$\tilde T^P(\gamma)(N)$ rooted at~$N$ in~$\tilde T^P(\gamma)$
 into a valid quasi-Farkas-certificate for~$L$ in the branch-and-bound
 tree~$\tilde T^P(\gamma)$ for~$P(z)$ relative to~$[0,1]^{n_1}$.
 \end{claim}

 Note that $\gamma$ is only partially defined by the definition
 given in the claim; however, all entries of $\gamma$ which
 are relevant for the validity of the quasi-Farkas-certificates
 in~$\tilde T^P(\gamma)(N)$ are defined.
 
 It suffices to show the claim, since then~$C \coloneqq C_r$
 (where~$r$ denotes the root of~$T'$) decides whether there exists a
 branch-and-bound tree~$(\tilde T)^P$ for~$P(z)$ conforming
 to~$\tilde T$ and hence separates~$Z_1$ and~$Z_2$.

 We begin by noting that for every input~$\gamma_M^\pm$ of such a
 circuit~$C_N$ corresponding to the right-hand-side of the disjunction
 used at a node~$M$ in~$\V(N)$, we have $C_N (z,\tilde{\gamma},L^M) = 1$ for all
 $(z,\tilde{\gamma})\in \R^{\ZZ\cup (\V(N)\setminus \gamma_M^\pm)}$, since then the
 side of the disjunction at~$M$ which corresponds to the branch
 containing~$N$ does not intersect~$[0,1]^{n_1}$ and hence all vectors
 attached to leaves in~$\tilde T^P(\gamma)(N)$ are valid
 quasi-Farkas-certificates.
 
 We prove the claim via induction on~$\card{\tilde T(N)}$. If
 $\card{\tilde T(N)} = 1$, then~$N$ is a leaf. Hence,
 $\U(N) = \emptyset$ and since we are given values for all variables
 from~$\ZZ \cup \V(N)$, we are given all right-hand-sides to the
 subproblem
 $\tilde T^P(\gamma)_N(P(z)) \eqqcolon \{x\in \R^{n_1}~|~ Ex \leq e\}$
 associated to the leaf~$N$ of the
 tree~$\tilde T^P(\gamma)$ for~$P(z)$.  We have to test whether~$f^N$
 is a valid quasi-Farkas-certificate.  To this end, it suffices to
 test if $(f^L)^\top e <0$, since we have~$(f^L)^\top E = 0$,
 because~$T'$ is a valid certified branch-and-bound tree for
 $P(z) \times Q(z)$ (see the proof of
 Lemma~\ref{lem:conformal_cert_interpolation}).  Note that
 $\tilde T^P(\gamma)_N(P(z))$ contains constraints which are also
 contained in~$P(z)$ and are indexed with numbers~$i\in [m_1]$ and
 constraints of the form $\alpha_M^\top x \geq \gamma_M+1$ or
 $\alpha_M^\top x \leq \gamma_M$ coming from branching, which we will
 index with the node~$M$ at which they appear in a disjunction.
  We then calculate:
  \begin{align*}
    (f^L)^\top e &= \sum_{i\in [m_1]} f^N_i(a_i-C_iz) + \sum_{\gamma^+_M\in \V^+(N)}f^N_M(-L^M_{\min}-\gamma^+_M-1)
                   + \sum_{\gamma^-_M\in \V^-(N)}f^N_M(L^M_{\max} - \gamma_M^-)\\
                 &\eqqcolon k^N - \sum_{\tau \in \ZZ\cup \V(N)} s^N_\tau \cdot \tau,
  \end{align*}
  where~$C_i$ is the~$i$-th row of~$C$ and the second line is defined by aggregating variables and constants. We note that the resulting~$s^N_\tau$ are non-negative (recall~$C$ is non-negative).
 Evidently, the sum in the second line can be computed by a monotone real
 circuit with inputs corresponding to the elements of
 $\ZZ\cup \V(N)$ and $\card{\ZZ\cup \V(N)}-1$ further gates by
 iteratively adding summands.  Note that adding~$s^N_\tau$ times
 the first input to the second input is a monotone operation, since~$s^N_\tau$ is non-negative. By post-composing the function
 applied at the output gate with the function sending numbers larger
 than~$k^N$ to~$1$ and numbers at most~$k^N$ to~$0$, we obtain a monotone
 real circuit~$\hat C_N$ that decides whether~$f^N$ is a
 Farkas-certificate.

 We modify~$\hat C_N$ to obtain a monotone real circuit~$C_N$ which
 decides whether~$f^N$ is a quasi-Farkas-certificate as follows: For
 every gate~$g$ in~$\hat C_N$ which adds~$s^N_\tau$-times the value of
 an input $\gamma_M^+ \in \V(N)$ (or~$\gamma_M^-$) to our sum, we
 modify the function applied at this gate such that it adds a very
 large constant~$K^N$ instead, if $\gamma_M^+ \geq L_M$
 ($\gamma_M^- \geq L_M$). By our definition of~$L_M$ and~$\gamma_M$,
 this is the case if and only if the side of the disjunction at the
 node~$M$ corresponding to the subtree of~$M$ containing~$N$, does not
 intersect~$[0,1]^{n_1}$, which makes~$f^N$ a valid
 quasi-Farkas-certificate relative to~$[0,1]^{n_1}$ by
 definition. Hence, if we choose~$K^N$ sufficiently large, such
 that~$C_N$ will certainly accept in this case, for example
 $K^N\coloneqq k^N+1$, then~$C_N$ correctly decides whether~$f^N$ is a
 valid quasi-Farkas-certificate. Moreover,~$C_N$ satisfies the claimed
 bound on its size.

 If~$\card{T'(N)} > 1$, we appeal to Corollary~\ref{cor:circ_bal}:
 Let~$N_\leq$ and~$N_\geq$ denote the children of~$N$. Then, by the
 induction hypothesis, there exist circuits~$C_{N_\leq}$
 and~$\smash{C_{N_\geq}}$ for these nodes as in the claim. Since
 $\V(N_\leq) = \V(N) \cup \{\gamma_N^-\} = (\V(N_\geq) \setminus
 \{\gamma_N^+\}) \cup \{\gamma_N^-\}$ and
 $\smash{C_{N_\leq} (z,\tilde{\gamma},L^N) = C_{N_\geq} (z,\tilde{\gamma},L^N)} = 1$ for all
 $(z,\tilde{\gamma})\in \smash{\R^{\ZZ\cup \V(N)}}$, we may apply
 Corollary~\ref{cor:circ_bal} to~$C_{N_\leq}$ and~$C_{N_\geq}$ (with
 $\smash{\Lambda_{\max} = L^N}$, $\Lambda_{\min} = 0$ and~$\kappa = L^N$), in a way
 which invokes the larger circuit only once.
 
 To see that~$C_N$ is no larger than claimed, assume the
 subtree~$T'(N_\leq)$ of~$T'$ rooted at~$N_\leq$ is smaller than the
 one rooted at~$N_\geq$, the other case is
 analogous. Hence,~$T'(N_\leq)$ has size at most~$\card{T'(N)}/2$
 while~$T'(N_\geq)$ has size at most~$\card{T'(N)}-1$.
    Then, compute
    \begin{align*}
      \card{C_N} & \leq \card{C_{N_\leq}} \cdot (\lceil \log(L^N+1) \rceil) + \card{C_{N_\geq}}\\
                 & \leq 
                   \card{T'(N_\leq)} \cdot 2(\card{T'}+n_3)\cdot(\lceil \log(L+1) \rceil)^{\log (\card{T'(N)}/2)} \cdot (\lceil \log(L^N+1) \rceil) \\
                 & \phantom{\leq}  +      
                   \card{T'(N_\geq)} \cdot 2(\card{T'}+n_3)\cdot(\lceil \log(L+1) \rceil)^{\log (\card{T'(N)} -1)} \\
                 & \leq 
                   \card{T'(N_\leq)} \cdot 2(\card{T'}+n_3)\cdot(\lceil \log(L+1) \rceil)^{\log \card{T'(N)}}\\
                 & \phantom{\leq}  +      
                   \card{T'(N_\geq)} \cdot 2(\card{T'}+n_3)\cdot(\lceil \log(L+1) \rceil)^{\log \card{T'(N)}} \\
                 & \leq 
                   \card{T'(N)} \cdot 2(\card{T'}+n_3)\cdot(\lceil \log(L+1) \rceil)^{\log \card{T'(N)}}.
    \end{align*}
    Finally, set~$C\coloneqq C_r$ for the circuit~$C_r$ given by the claim for the root node~$r$ of~$T'$ and note $\lceil \log(L+1) \rceil = \lceil \log (n(10n^2)^{(n+2)^2}+3) \rceil$ as well as $\card {T'} \leq (4n+5) \card T$. Hence~$C_r$ has size at most
    \begin{align*}
      & \card{T'(r)} \cdot 2(\card{T'}+n_3)\cdot(\lceil \log(L+1) \rceil)^{\log \card{T'(r)}}\\
      & \leq (4n+5) \card T \cdot 2 [(4n+5)\card T +n]\cdot [(n+2)^2\log(10n^3+3)]^{\log((4n+5)\card T)}\\
      & \leq 2(5n+5)^2\card{T}^2 \cdot  [(n+2)^2\log(10n^3+3)]^{\log((4n+5)\card T)}\\
      & \leq 50(n+1)^2\card{T}^2 \cdot [(n+2)^2\log(10n^3+3)]^{\log((4n+5)\card T)}.\qedhere
    \end{align*} 
  \end{proof}

  The computations bounding the circuit size in the recursive step are
  taken from Fleming et al.~\cite{fleming2021power} where they are
  used to show that branch-and-bound with really small coefficients
  can be quasi-polynomially simulated by cutting planes. A very similar recursive formula
  already appears in~\cite{beame1996simplified}, where it is used to
  show that branch-and-bound for variable disjunctions is
  quasi-automatizable.

  \subsection{Proof of Theorems~\ref{thm:bb_lb_BMS} and~\ref{thm:3CNF_hard}}
  
  \begin{proof}[Proof of Theorem~\ref{thm:bb_lb_BMS}]
    Assume that we have a family of branch-and-bound trees~$T$
    for~(\ref{eq:BMS_alternative}), one for each~$r$, such that
    $\smash{\card{T} \in 2^{O(n^{1/6-\epsilon})}}$ for some~$\epsilon >
    0$. Then Theorem~\ref{thm:qfeas_real_monotone_interpolation} gives
    rise to a family of circuits~$C_n$ separating the CC-pair of size
    \begin{align*}
      \card{C_n}
      & = 50(n+1)^2\card{T}^2 \cdot [(n+2)^2\log(10n^3+3)]^{\log((4n+5)\card T)} \\
      & = 50(n+1)^2 2^{2O(n^{1/6-\epsilon})} \cdot \left(2^{\log[(n+2)^2\log(10n^3+3)]}\right)^{\log(4n+5)O(n^{1/6-\epsilon})} \\
      & = 2^{O(n^{1/6-\epsilon})} \cdot 2^{\log[(n+2)^2\log(10n^3+3)]\cdot\log(4n+5)O(n^{1/6-\epsilon})}\\
      & = 2^{O(n^{1/6-\epsilon})}\cdot 2^{O(n^{\epsilon/2}) \cdot O(n^{1/6-\epsilon})}  
      = 2^{O(n^{1/6-\epsilon})}\cdot 2^{O(n^{1/6-\epsilon/2})} = 2^{O(n^{1/6-\epsilon/2})} .
    \end{align*}
    Since we have
    $n = n_1 +n_2+n_3 =  r\lfloor \frac 1 8 (r/\log r)^{2/3}\rfloor +r+(r^2-r)/2$
    and~$n_3 = r+(r^2-r)/2$
    we have
    \[
      1\leq \frac{n}{n_3} = 1 +  \frac{O(r^{5/3})}{\Omega(r^2)}  
    \]
    Since~$n\rightarrow \infty$ implies~$r\rightarrow \infty$,
    we have~$n_3 \in \Theta(n)$. But then we have
  \[\card{C_n}\in 2^{O(n^{1/6-\epsilon/2})} = 2^{O(n_3^{1/6-\epsilon/2})},\]
     which contradicts Theorem~\ref{thm:mon_circ_lb_BMS}.
   \end{proof}

   For the proof of Theorem~\ref{thm:3CNF_hard}, we require an analog of
   Theorem~8 in~\cite{hrubevs2017random}.

   Given an (unsatisfiable) CNF $\C = \{C_1,\dots,C_m\}$ and a partition of its variables $X_0\cup X_1$,
   let
   $Y_1$, $Y_2$ and $\D$
   be defined as in Section~\ref{sec:prelim}.

   \begin{obs}\label{obs:split_easy}
     Every branch-and-bound tree for the ILP~(\ref{eq:CNF_ILP}) for $\C$ and
     any partition~$X_0\cup X_1$  is also a branch-and-bound tree
     for the ILP~(\ref{eq:CNF_ILP}) for $\D$.
   \end{obs}
   \begin{proof}
     It suffices to note that linear constraints corresponding to the original clauses of
     $\C$ are valid inequalities for the LP-relaxation of~(\ref{eq:CNF_ILP}) for $\D$.
   \end{proof}

   \begin{lemma}\label{lem:inf_cert_interpol}
     For every branch-and-bound tree~$T$ for~(\ref{eq:CNF_ILP}) for $\C$ and
     a partition~$X_0\cup X_1$ of its variables, there is a monotone real circuit computing an
     $(X_0,X_1)$-infeasibility certificate for $\C$ of size quasi-polynomial in $n$, $m$ and $\card T$,
     i.e., size at most~$\poly(n+m+ f(n))^{\log(n+m+f(n))}$.
   \end{lemma}
   
   \begin{proof}
     By Observation~\ref{obs:split_easy}, we can consider~$T$ as a
     branch-and-bound tree for~(\ref{eq:CNF_ILP}) for $\D$ and hence can
     apply Theorem~\ref{thm:qfeas_real_monotone_interpolation} to
     obtain a monotone real circuit separating $Y_0$ and $Y_1$ of
     size
     \[ 50(n'+1)^2\card{T}^2 \cdot [(n'+2)^2\log(10n'^3+3)]^{\log((4n'+5)\card T)} \in \poly(n+m+ T)^{\log(n+m+\card T)},
     \]
     where $n' =2m+n$.
       Since a monotone function separating $Y_0$ and $Y_1$ is an
       $(X_0,X_1)$-infeasibility certificate for $\C$, the lemma is
       shown.
   \end{proof}

   Finally, combining Theorem~\ref{thm:CNF_cert_compl} with Lemma~\ref{lem:inf_cert_interpol}, we obtain a proof for Theorem~\ref{thm:3CNF_hard}:

   \begin{proof}[Proof of Theorem~\ref{thm:3CNF_hard}]
     Assume that there exists a function $f\in O(2^{n^{o(1)}})$ such
     that for a random $k$-CNF $\C$
     with $O(n2^k)$ clauses and $2n$ variables there
     exists a branch-and-bound tree~$T$ refuting~(\ref{eq:CNF_ILP})
     for~$\C$ of size at most $f(n)$ with non-negligible probability,
     i.e., the probability of this occurring does not tend to $0$ for
     $n\rightarrow \infty$.  Then, due to
     Lemma~\ref{lem:inf_cert_interpol}, for any fixed
     partition~$X_0 \cup X_1$ of the variables with
     $\card{X_0} = \card{X_1}= n$ there is a monotone real circuit
     computing an $(X_0,X_1)$-certificate for $\C$ of
     size~$\poly(n+m+ f(n))^{\log(n+m+f(n))}$ with non-negligible
     probability. We may assume that $f(n) \geq \max(n,m)$ for
     simplicity, hence $\C$ has size at
     most~$g\in\poly(f(n))^{\log(f(n))}$. However, clearly
     $\smash{\poly(O\big(2^{n^{o(1)}}\big)) = O\big(2^{n^{o(1)}}\big)}$ and
     \[
       O\big(2^{n^{o(1)}}\big)^{\log(O(2^{n^{o(1)}}))}
       = O\big(2^{n^{o(1)}}\big)^{O({n^{o(1)}})} = O\big(2^{n^{o(1)}\cdot O(n^{o(1)})}\big)
       =O\big(2^{n^{o(1)}}\big) .\]
     Hence, for a random $k$-CNF
     with $O(n2^k)$ clauses and $2n$ variables with partition
     $X_0 \cup X_1$ such that $\card{X_0} = \card{X_1}= n$ and
     $k \geq c \log (n)$ there is an $(X_0,X_1)$-certificate with size
     at most $g(n)$ with non-negligible probability which contradicts
     Theorem~\ref{thm:CNF_cert_compl}.
   \end{proof}

\bibliographystyle{splncs04}
\bibliography{not_yet_feas_int}

\end{document}